\documentclass{amsart}
\usepackage[cp1251]{inputenc}
\usepackage[english]{babel}
\usepackage{amsmath}
\usepackage{amsfonts,amssymb}
\usepackage{mathrsfs}
\usepackage{bm}
\usepackage{esint}
\usepackage{amsthm}
\usepackage{tikz-cd}
\usepackage{a4wide}

\usepackage{hyperref}

\makeatother

\numberwithin{equation}{section}

\newtheorem{theorem}{Theorem}[section]
\newtheorem{corollary}[theorem]{Corollary}
\newtheorem{lemma}[theorem]{Lemma}
\newtheorem{proposition}[theorem]{Proposition}
\newtheorem{definition}[theorem]{Definition}

\newcommand{\N}{\mathbb{N}}
\newcommand{\Q}{\mathbb{Q}}
\newcommand{\R}{\mathbb{R}}

\newcommand{\restr}[1]{\lower3pt\hbox{$|_{#1}$}}
\newcommand{\pr}{\mathscr P}

\newcommand{\X}{{\rm X}}
\newcommand{\Y}{{\rm Y}}

\newcommand{\sfd}{{\sf d}}
\newcommand{\mm}{{\mathfrak m}}

\renewcommand{\L}{{\rm L}}

\newcommand{\LIP}{{\rm Lip}}

\newcommand{\RCD}{{\sf RCD}}

\renewcommand{\d}{{\rm d}}

\newcommand{\lip}{{\rm lip}}
\newcommand{\Lip}{{\rm Lip}}

\newcommand{\weakto}{\rightharpoonup}
\newcommand{\nchi}{{\raise.3ex\hbox{$\chi$}}}
\newcommand{\supp}{{\rm supp}}
\newcommand{\lims}{\varlimsup}
\newcommand{\limi}{\varliminf}
\newcommand{\eps}{\varepsilon}

\DeclareMathOperator*{\esssup}{\rm ess-sup}

\newcommand{\Fl}{{\sf Fl}}
\renewcommand{\div}{{\rm div}}

\newcommand{\KS}[1]{{\sf KS}_{#1}^p}
\newcommand{\KSt}[1]{{\sf KS}_{#1}^2}

\newcommand{\la}{\langle}
\newcommand{\ra}{\rangle}
\newcommand{\CAT}{{\sf CAT}}

\begin{document}
\author{Nicola Gigli}
\author{Alexander Tyulenev}
\address{SISSA, Trieste}
\email{ngigli@sissa.it, tyulenev-math@yandex.ru }

\title{Korevaar-Schoen's directional energy and Ambrosio's regular Lagrangian flows}

\maketitle

\begin{abstract}
We develop Korevaar-Schoen's theory of directional energies for metric-valued Sobolev maps in the case of $\RCD$  source spaces;  to do so we crucially rely on Ambrosio's concept of Regular Lagrangian Flow.

Our review of Korevaar-Schoen's spaces brings new (even in the smooth category) insights on some aspects of the theory, in particular concerning the notion of `differential of a map along a vector field' and about the parallelogram identity for $\CAT(0)$ targets. To achieve these, one of the ingredients we use is a  new (even in the Euclidean setting) stability result for Regular Lagrangian Flows.
\end{abstract}

\tableofcontents

\section{Introduction}
In the seminal paper \cite{ES64}, Eells-Sampson proved Lipschitz regularity of harmonic maps from a manifold $M$ to a simply connected manifold $N$ with non-positive sectional curvature, the estimate being in term of a bound from below on the Ricci curvature and an upper bound on the dimension of the source manifold. A crucial step in their argument is the proof of the now-called Bochner-Eells-Sampson inequality, namely:
\begin{equation}
\label{eq:BES}
\Delta\frac{|\d u|_{\sf HS}^2}{2}\geq \d u(\Delta u)+K|\d u|^2_{\sf HS}
\end{equation}
valid for smooth maps $u:\Omega\subset M\to N$ provided ${\rm Ric}_M\geq K$. From this one sees that if $u$ is harmonic then
\begin{equation}
\label{eq:BES2}
\Delta\frac{|\d u|_{\sf HS}^2}{2}\geq K|\d u|^2_{\sf HS}
\end{equation}
and then a Moser's iteration argument gives that $|\d u|_{\sf HS}$ is locally bounded from above (the upper dimension bound comes into play in the constants appearing in this process), which was the claim.

Given that the final estimate does not depend on the smoothness of $M,N$ but only in the stated curvature bounds, it is natural to wonder whether such smoothness can be removed. This problem attracted the attention of several mathematicians, see in particular \cite{GS92}, \cite{KS93}, \cite{ZZ18} and the survey \cite{Hajlasz2009} for an overview on this and related topics. We remark that given the kind of assumptions in Eells-Sampson work, the natural non-smooth class of spaces for which such Lipschitz regularity is expected to hold is that of $\RCD(K,N)$ spaces as source and $\CAT(0)$ ones as target; so far this generality has been out of reach.

This paper is part of a bigger project aiming at reproducing \eqref{eq:BES} in such fully synthetic setting, see also \cite{GPS18} and \cite{DMGSP18} for other contributions in this direction. The purpose of the current manuscript is to generalize part of Korevaar-Schoen's theory in \cite{KS93} to the case of source spaces which are  $\RCD$. Specifically, one of the definitions proposed in \cite{KS93} is that of  `map from a smooth manifold to a metric space which is Sobolev along a given direction': we adapt this construction to the case of $\RCD$ source and postpone to a future contribution the study of what in \cite{KS93} has been called `total energy functional'. Our main results here are:
\begin{itemize}
\item[i)] We obtain new stability results for Regular Lagrangian Flows both on $\RCD$ spaces and in the Euclidean setting, see Theorems \ref{thm:stabfl}, \ref{thm:stabfleu}.
\item[ii)] We reproduce the theory of what we call Korevaar-Schoen (Sobolev) space relying on the aforementioned concept of Regular Lagrangian Flow. In particular we introduce the Korevaar-Schoen space $\KS Z(\Omega,\Y)$ of maps from $\Omega\subset \X$ to  $\Y$ which are Sobolev along the vector field $Z$, and for $u\in \KS Z(\Omega,\Y)$ we define the quantity $|\d u(Z)|$ which plays the role of the modulus of the differential of $u$ applied to $Z$ (and corresponding to the $p$-th root of the directional energy density in \cite{KS93}). See Section \ref{se:ks}.
\item[iii)] Using our stability result for RLF we prove the `triangle inequality'
\[
|\d u(\alpha_1 Z_1+\alpha_2Z_2)|\leq|\alpha_1 |\, |\d u(Z_1)|+|\alpha_2|\,|\d u(Z_2)|
\]
\item[iv)] We show that for $u\in \KS Z(\Omega,\Y)$ not only the quantity $|\d u(Z)|$ is well defined, but also the differential $\d u(Z)$ of $u$ applied to $Z$ makes sense, see Definition \ref{def:duz}.
\item[v)] Using the previous point and a duality argument we show that under some kind of Sobolev  condition on the target space $\Y$, we also have the parallelogram identity
\begin{equation}
\label{eq:paridintro}
|\d u(Z_1+Z_2)|^2+|\d u(Z_1-Z_2)|^2=2\big(|\d u(Z_1)|^2+|\d u(Z_2)|^2\big),
\end{equation}
see Theorem \ref{thm:parid}. According to \cite{DMGSP18}, $\CAT(0)$ spaces have the required condition.
\end{itemize}
Let us briefly comment the above results. In $(i)$  the relevant notion of convergence of the underlying vector fields is that of `weak convergence in time and strong in space', see Definition \ref{def:wtss}. Previous results in this direction (see \cite[Remark 5.11]{Ambrosio2008}) required quantitative estimates on the regularity of the flows which are not available  neither in the $\RCD$ setting (but see \cite{BS18a}, \cite{BS18}) nor in the Euclidean one for $BV$ vector fields (but see \cite{BB17}). More recent contributions \cite{ACF15} avoid the use of such  quantitative estimates, but still our setting was not covered.

To the expert's eye, what claimed in $(ii)$ is perhaps not so surprising, as it is by now clear that the concept of Regular Lagrangian Flow provides the correct replacement for the notion of flow of a vector field in a non-smooth environment. Even so, let us mention that our presentation offers some (marginal) improvement w.r.t.\ the original one in \cite{KS93}, see in particular Definition \ref{def:mwug} and compare to the original proof of the absolute continuity of the directional energy densities.

For what concerns the triangle inequality mentioned  in $(iii)$, we can obtain it under the only assumption that the map $u$ is in $\KS {Z_1}(\Omega,\Y)\cap\KS {Z_2}(\Omega,\Y)$, without needing a control of the total energy as in  \cite{KS93} (and indeed we won't mention total energy at all in this manuscript). In particular, even  in the case of smooth source space,  our result strengthens previously existing ones. This is possible thanks to a kind of Trotter-Kato formula for Regular Lagrangian Flows that we obtain as a corollary of  the stability results in $(i)$, see Proposition \ref{prop:sumvect}.

The definition in $(iv)$ makes use of the theory of $L^0$-normed modules as tools to develop first-order calculus on metric measure spaces as proposed in \cite{Gigli14}. More in detail, our approach for defining $\d u(Z)$ should be seen as an adaptation to the current framework of the recent construction of differential of a metric-valued Sobolev map proposed in \cite{GPS18}.

Finally, the proof of the parallelogram identity \eqref{eq:paridintro} is perhaps what conceptually differs the most from the approach in \cite{KS93}. Indeed, in  \cite{KS93} it is observed that $\CAT(0) $ spaces have a sort of metric parallelogram (in)equality and this information is directly exploited to obtain \eqref{eq:paridintro}; here, instead, in some sense we decouple the study of the geometry of the target from the one of Sobolev maps valued in it. More precisely, thanks to the existence of a sort of linear differential (point $(iv)$) we can easily prove that if the target space $(\Y,\sfd_\Y)$ is so that `for any Radon measure $\mu$ on it the Sobolev space $W^{1,2}(\Y,\sfd_\Y,\mu)$ is Hilbert' (see Definition \ref{def:uih}), then necessarily \eqref{eq:paridintro} holds. Thus here the discussion is fully at the `Sobolev' level. The question is then evidently whether there are spaces $\Y$ as above, and in particular if $\CAT(0)$ spaces have this property: the non-trivial affirmative answer (valid more generally for locally $\CAT(k)$ spaces) has been obtained in \cite{DMGSP18}.

\bigskip

We conclude this introduction remarking that it seems impossible to obtain the desired Lipschitz regularity of harmonic maps from $\RCD$ to $\CAT(0)$ spaces fully mimicking the approach in \cite{KS93}. The problem is that Regular Lagrangian Flows are not Lipschitz in general (because typically there are not Lipschitz vector fields on $\RCD$ spaces) and as such they cannot be used in the same spirit as in \cite{KS93} to provide any kind of Euler's equation for our minimizers of the energy functional. This is one of the reasons that led to the attempt of establishing the `full' inequality \eqref{eq:BES} rather than focussing `only' on its version for harmonic maps \eqref{eq:BES2}.

\bigskip

\noindent{\bf Acknowledgments} This research has been supported by the MIUR SIR-grant `Nonsmooth Differential Geometry' (RBSI147UG4).

\section{Preliminaries}
\subsection{Sobolev calculus}\label{se:sobcal}

To keep the presentation short we assume that the reader is familiar with the concept of Sobolev functions on a metric measure space (\cite{Cheeger00}, \cite{Shanmugalingam00}, \cite{AmbrosioGigliSavare11}, \cite{AmbrosioGigliSavare11-3}), with that of $L^0$-normed modules and differentials of real valued Sobolev maps  and with second order calculus on $\RCD$ spaces (\cite{Gigli14}, \cite{Gigli17}).

Here we only recall those concepts we shall use most frequently. For a generic metric space $(\X,\sfd)$ we shall denote by $\Lip(\X),\Lip_{bs}(\X)$ the spaces of real valued Lipschitz functions and Lipschitz functions with bounded support, respectively. The Lipschitz constant of $f\in \Lip(\X)$ will be denoted $\Lip(f)\in[0,+\infty)$;  the local Lipschitz constant $\lip f:\X\to[0,\infty)$ is defined as
\[
\lip f(x):=\lims_{y\to x}\frac{|f(y)-f(x)|}{\sfd_\X(x,y)}\quad\text{ if $x$ is not isolated, $0$ otherwise.}
\]
We shall most often work with a metric measure space $(\X,\sfd,\mm)$ which is  complete and separable as metric space and equipped with a non-negative and non-zero Radon measure giving finite mass to bounded sets.
\begin{definition}[The Sobolev space $W^{1,p}(\X,\sfd,\mm)$]\label{def:sob}
Let $p \in (1,\infty)$ and $f\in L^p(\X)$. We say that  $f\in W^{1,p}(\X)$ provided there is a function $G\in L^p(\mm)$ and a sequence $(f_n)\subset \LIP_{bs}(\X)$ converging to $f$ in $L^p(\mm)$  such that $(\lip(f_n))$ weakly converges to $G$ in $L^p(\mm)$.
\end{definition}
For $f\in W^{1,p}(\X)$ we recall that there is a minimal, in the $\mm$-a.e.\ sense, non-negative function $G\in L^p(\mm)$ for which the situation in Definition \ref{def:sob} occurs: it will be denoted $|D f|$ and called minimal weak upper gradient. One can check that
\begin{equation}
\label{eq:optlip}
\text{$\forall f\in W^{1,p}(\X)$ there is $(f_n)\subset \LIP_{bs}(\X)$ converging to $f$ in $L^{p}(\mm)$ such that $\lip(f_n)\to |D f|$ in $L^p(\mm)$.}
\end{equation}
Now suppose that $\mm'$ is another Radon measure on $\X$ giving finite mass to bounded sets and such that $\mm\leq \mm'$. Then, given $p \in (1,\infty)$, it is clear that $L^p(\mm')\subset L^p(\mm)$ with continuous inclusion, thus a direct consequence of Definition \ref{def:sob} above and of minimal weak upper gradient is that
\begin{equation}
\label{eq:gradord}
W^{1,p}(\X,\sfd,\mm')\subset W^{1,p}(\X,\sfd,\mm)\qquad\text{ and }\qquad |D_\mm f|\leq |D_{\mm' }f|\quad\mm-a.e.\ \forall f\in W^{1,p}(\X,\sfd,\mm'),
\end{equation}
where with $|D_\mm f|$, $|D_{\mm' }f|$ we denoted the minimal weak upper gradients in $W^{1,p}(\X,\sfd,\mm)$, $W^{1,p}(\X,\sfd,\mm')$ respectively.

\bigskip

The concept of $L^0(\mm)$-normed module is introduced in order to `extract' a notion of differential from that of minimal weak upper gradient:
\begin{theorem}[Cotangent module and differential]\label{thm:defdif} With the above notation and assumptions, there is a unique (up to unique isomorphism) couple $(L^0(T^*\X),\d)$ with $L^0(T^*\X)$ being a $L^0(\mm)$ normed module, $\d:W^{1,2}(\X)\to L^0(T^*\X)$ linear and such that: $|\d f|=|Df|$ $\mm$-a.e.\ for every $f\in W^{1,2}(\X)$ and $\{\d f:f\in W^{1,2}(\X)\}$ generates $L^0(T^*\X)$.
\end{theorem}
Among the various constructions related to $L^0$-normed modules, we shall make use of the one of pullback:
\begin{theorem}[Pullback]\label{thm:pb}
Let $(\X,\sfd_\X,\mm_\X),(\Y,\sfd_\Y,\mm_\Y)$ be metric measure spaces as above, $u:\X\to\Y$ a Borel map such that $u_*\mm_\X\ll\mm_\Y$ and $\mathscr M$ an $L^0(\mm_\Y)$-normed module. Then there is a unique (up to unique isomorphism) couple $(u^*\mathscr M,[u^*])$ such that $u^*\mathscr M$ is a $L^0(\mm_\X)$-normed module and $[u^*]:\mathscr M\to u^*\mathscr M$ is linear, continuous and such that $|[u^*v]|=|v|\circ u$ $\mm_\X$-a.e.\ for every $v\in\mathscr M$ and $\{[u^*v]:v\in\mathscr M\}$ generates $u^*\mathscr M$.
\end{theorem}
The module $u^*\mathscr M$ is called the pullback module and $[u^*]$ the pullback map. These   can also be characterized by the following universal property:
\begin{proposition}[Universal property of the pullback]\label{prop:univprop}
With the same notation and assumptions as in Theorem \ref{thm:pb} above, let $V\subset \mathscr M$ a generating subspace, $\mathscr N$ a $L^0(\mm_\X)$-normed module and $T:V\to \mathscr N$ a linear map such that $ |T(v)|\leq f|v|\circ u$ $\mm_\X$-a.e. $\forall v\in V$ for some nonnegative $\mm_\X$-a.e. function $f \in L^0(\mm_\X)$. Then there exists a unique $L^0(\mm_\X)$-linear and continuous map $\tilde T:u^*\mathscr M\to\mathscr N$ such that $\tilde T([u^*v])=T(v)$ for every $v\in V$ and this map satisfies
\begin{equation}
\label{eq:pb1}
 |\tilde T(w)|\leq f | w|\quad\mm_\X-a.e.\qquad\forall w\in u^*\mathscr M.
\end{equation}
\end{proposition}
These properties of pullbacks have been studied in \cite{Gigli14}, \cite{Gigli17} for maps satisfying $u_*\mm_\X\leq C\mm_\Y$; the generalization to the case of $L^0$-normed modules has been considered in \cite{GR17} and \cite{Benatti18}.

\bigskip

Finally, let us present the construction of the `extension functor'. Informally speaking, it might happen that one deals with a measure $\mu\ll\mm$ and with a $L^0(\mu)$-normed module $\mathscr M$ and would like to think  $\mathscr M$ as $L^0(\mm)$-normed module, where its elements are 0 on those regions which $\mu$ does not see. The extension functor formalizes this construction.

Thus let $\mathscr M$ be a $L^0(\mu)$-normed module with $\mu\ll\mm$. Notice that we have a natural projection/restriction operator ${\rm proj}:L^0(\mm)\to L^0(\mu)$ given by passage to the quotient up to equality $\mu$-a.e.\ and a natural right inverse of it, namely   an `extension' operator ${\rm ext}:L^0(\mu)\to L^0(\mm)$ which sends $f\in L^0(\mu)$ to the function equal to $f$ $\mm$-a.e.\ on $\{\frac{\d\mu}{\d\mm}>0\}$ and to $0$ on $\{\frac{\d\mu}{\d\mm}=0\}$.

Then  we put ${\rm Ext}(\mathscr M):=\mathscr M$ as set,  define the multiplication of $v\in {\rm Ext}(\mathscr M) $ by $f\in L^0(\mm)$ as ${\rm proj}(f)v\in \mathscr M={\rm Ext}(\mathscr M)$ and the pointwise norm as ${\rm ext}(|v|)\in L^0(\mm)$. We shall denote by ${\rm ext}:\mathscr M\to {\rm Ext}(\mathscr M)$ the identity map and notice that in a rather trivial way we have
\[
{\rm Ext}(\mathscr M^*)\sim {\rm Ext}(\mathscr M)^*\qquad\text{ via the coupling }\qquad{\rm ext}(L)\big({\rm ext}(v)\big):={\rm ext}(L(v)).
\]
In what follows we shall  always implicitly make this identification.

\bigskip

For the definition of $\RCD(K,\infty)$ space see \cite{AmbrosioGigliSavare11-2} and for the second order calculus see \cite{Gigli14}. Here we just recall that the space ${\rm Test}(\X)$ of test functions (introduced in \cite{Savare13}) is defined as
\[
{\rm Test}(\X):=\Big\{f\in L^\infty\cap W^{1,2}(\X)\ :\ f\in D(\Delta),\ |D f|\in L^\infty(\X),\ \Delta f\in W^{1,2}(\X)\Big\},
\]
and that the space of ${\rm TestV}(\X)$ of test vector fields is defined as
\[
{\rm TestV}(\X):=\Big\{\sum_{i=1}^ng_i\nabla f_i\ :\ n\in\N,\ f_i,g_i\in {\rm Test}(\X)\quad\forall i=1,\ldots,n\Big\}.
\]
For the definition of covariant derivative - defined through integration by parts - and the space $W^{1,2}_C(T\X)$ of $L^2$ vector fields with covariant derivative in  $L^2$ see \cite{Gigli14}. Recall that ${\rm TestV}(\X)\subset W^{1,2}_C(T\X)$ and thus so does its $W^{1,2}_C$-closure, which is denoted $H^{1,2}_C(T\X)$ (it is an open problem to understand whether $H^{1,2}_C(T\X)=W^{1,2}_C(T\X)$ or not). In particular for $f\in {\rm Test}(\X)$ we have $\nabla f\in L^\infty\cap H^{1,2}_C(T\X)$. Also, the identity $\nabla(\nabla f)={\rm Hess}(f)$ holds, having freely identified tangent and cotangent modules via the Riesz isomorphism.

Finally recall the following form of Leibniz rule: for $v\in L^\infty\cap H^{1,2}_{C}(T\X)$ and $w\in L^\infty \cap W^{1,2}_{C}(T\X)$ we have
\begin{equation}
\label{eq:leibgrad}
\la v,w\ra\in W^{1,2}(\X)\qquad\text{ and }\qquad \d \la v,w\ra=\la\nabla_\cdot v,w\ra+\la v,\nabla_\cdot w\ra.
\end{equation}

\subsection{Sobolev and absolutely continuous curves}

We recall here some basic properties of Sobolev and absolutely continuous curves with values in a metric space.

Throughout this section, $(\Y,\sfd_\Y)$ will be a complete metric space.

\begin{definition}[Absolutely continuous curves]
A curve $\gamma:[0,T]\to \Y$ is said to be absolutely continuous provided there is $f\in L^1((0,T))$ such that
\begin{equation}
\label{eq:ac}
\sfd_\Y(\gamma_s,\gamma_t)\leq \int_t^sf(r)\,\d r\qquad\forall t,s\in[0,T],\ t<s.
\end{equation}
For  $p\geq 1$, the space $AC^p([0,T],\Y)$ consists of those absolutely continuous curves for which we can find $f$ as above in the space $L^p((0,T))$. Finally, $AC^p_{loc}([0,T),\Y)$ is the collection of all those curves which are in $AC^p(I,\Y)$ for every compact interval $I\subset [0,T)$.
\end{definition}
The following result is well-known; its proof can be found e.g.\ in \cite[Theorem 1.1.2]{AmbrosioGigliSavare08}:
\begin{theorem}[Metric speed]\label{thm:ms}
Let $p\geq 1$ and $\gamma\in AC^p([0,T],\Y)$. Then for a.e.\ $t\in[0,T]$ there exists the limit
\[
|\dot\gamma_t|:=\lim_{h\to 0}\frac{\sfd_\Y(\gamma_{t+h},\gamma_t)}{|h|},
\]
it defines a function in $L^p((0,T))$ and is the least - in the a.e.\ sense - function $f$ for which \eqref{eq:ac} holds.
\end{theorem}
We now turn to the definition of Sobolev curves and in order to do so we begin by spending few words on metric-valued $L^p$ spaces. Let $(\X,\sfd,\mm)$ be a metric measure space as before (i.e.\ complete, separable and with $\mm$ finite on bounded sets) and $(\Y,\sfd_\Y,\bar y)$ a pointed complete space.

For $p\geq 1$ the space $L^p(\X,\Y_{\bar y})$ consists of those (equivalence class up to $\mm$-a.e.\ equality) Borel maps $u:\X\to \Y$ which are essentially separably valued, i.e.\ for some negligible set $\mathcal N\subset\X$ we have that $u(\X\setminus\mathcal N)\subset\Y$ is separable, and satisfying $\int\sfd_\Y^p(u(x),\bar y)\,\d\mm<\infty$. If $\mm(\X)<\infty$ the particular choice of $\bar y$ is irrelevant and the reference to it will be omitted and if $\Y$ is a Banach space we shall always pick $\bar y=0$ and, again, omit the choice from the notation.

The space $L^p(\X,\Y_{\bar y})$ is equipped with the distance
\[
\sfd_{L^p(\X,\Y)}(u,v):=\left(\int\sfd_\Y^p(u,v)\,\d\mm\right)^{\frac 1p}.
\]
It is easy to see that $L^p(\X,\Y_{\bar y})$ is complete w.r.t.\ this distance and separable if $(\Y,\sfd_\Y)$ is so. Moreover, arguing as for the study of so-called `strong measurability' of Banach-valued functions, it is not hard to check that
\begin{equation}
\label{eq:simpledense}
\text{the collection of simple maps in $L^p(\X,\Y_{\bar y})$ is dense in $L^p(\X,\Y_{\bar y})$,}
\end{equation}
where `simple' means `attaining a finite number of values'.

\bigskip

We now come back to the study of Sobolev curves and consider the above construction for $\X:=[0,T]$, $T > 0$ equipped with the canonical distance and measure.

For $\gamma:[0,T]\to \Y$ Borel and $\eps\in(0,T)$ we define  $\mathcal E_{p,\eps}(\gamma)\in[0,\infty]$ as
\[
\mathcal E_{p,\eps}(\gamma):=\int_0^{T-\eps}\frac{\sfd_\Y^p(\gamma_{t+\eps},\gamma_{t})}{\eps^p}\,\d t
\]

\begin{definition}[Sobolev curves]
Let $(\Y,\sfd_\Y)$ be a complete space and $p\in(1,\infty)$. Given $T > 0$, the space of Sobolev curves $W^{1,p}([0,T],\Y)$ consists of those $\gamma\in L^p([0,T],\Y)$ Borel such that
\[
\mathcal E_p(\gamma):=\lims_{\eps\downarrow0}\mathcal E_{p,\eps}(\gamma)<\infty.
\]
\end{definition}
In order to study the properties of Sobolev curves, let us first study the functionals $\mathcal E_{p,\eps}$:
\begin{lemma}[Basic properties of $\mathcal E_{p,\eps}$]\label{le:baseep}
Let $\gamma\in L^p([0,T], \Y)$. Then:
\begin{itemize}
\item[i)] We have
\[
\lim_{\eps\downarrow0}\eps\big(\mathcal E_{p,\eps}(\gamma)\big)^{1/p}=0.
\]
\item[ii)] Let  $\eps\in(0,T)$ and $\lambda_i\in[0,1]$, $i=1,\ldots,n$, $n\in\N$, with $\sum_i\lambda_i=1$. Then
\begin{equation}
\label{eq:subpar}
\big(\mathcal E_{p,\eps}(\gamma)\big)^{1/p}\leq\sum_{i=1}^n\lambda_i\big(\mathcal E_{p,\lambda_i\eps}(\gamma)\big)^{1/p}.
\end{equation}
\end{itemize}
\end{lemma}
\begin{proof}\ \\
\noindent{\bf (i)} For $\eps\in(0,T)$ consider the map $T_\eps:L^p([0,T],\Y)\to \R$ given by $T_\eps(\gamma):=\eps (\mathcal E_{p,\eps}(\gamma))^{1/p}=\|\sfd_\Y(\gamma_\cdot,\gamma_{\cdot+\eps})\|_{L^p([0,T-\eps],\R)}$. Then we have
\[
\begin{split}
|T_\eps(\gamma)-T_\eps(\tilde\gamma)|&=\big|\|\sfd_\Y(\gamma_\cdot,\gamma_{\cdot+\eps})\|_{L^p([0,T-\eps],\R)}-\|\sfd_\Y(\tilde\gamma_\cdot,\tilde\gamma_{\cdot+\eps})\|_{L^p([0,T-\eps],\R)}\big|\\
&\leq \|\sfd_\Y(\gamma_\cdot,\gamma_{\cdot+\eps})-\sfd_\Y(\tilde\gamma_\cdot,\tilde\gamma_{\cdot+\eps})\|_{L^p([0,T-\eps],\R)}.
\end{split}
\]
Noticing that the triangle inequality on $\Y$ gives $|\sfd_\Y(a,b)-\sfd_\Y(c,d)|\leq \sfd_\Y(a,c)+\sfd_\Y(b,d)$ for every $a,b,c,d\in\Y$ and using again the triangle inequality in $L^p([0,T-\eps],\R)$ we then have
\[
\begin{split}
\|\sfd_\Y(\gamma_\cdot,\gamma_{\cdot+\eps})-\sfd_\Y(\tilde\gamma_\cdot,\tilde\gamma_{\cdot+\eps})\|_{L^p([0,T-\eps],\R)}&\leq\|\sfd_\Y(\gamma_\cdot,\tilde\gamma_{\cdot})\|_{L^p([0,T-\eps],\R)}+
\|\sfd_\Y(\gamma_{\cdot+\eps},\tilde\gamma_{\cdot+\eps})\|_{L^p([0,T-\eps],\R)}\\
&\leq 2\|\sfd_\Y(\gamma_\cdot,\tilde\gamma_{\cdot})\|_{L^p([0,T],\R)},
\end{split}
\]
i.e.\ $|T_\eps(\gamma)-T_\eps(\tilde\gamma)|\leq 2\sfd_{L^p([0,T],\Y)}(\gamma,\tilde\gamma) $. This shows that the maps $T_\eps$ are equiLipschitz, hence to conclude the proof it is sufficient to find a dense subset of $L^p([0,T],\Y)$ such that for any $\gamma$ in this subset it holds $T_\eps(\gamma)\to 0$ as $\eps\downarrow0$. It is readily checked that simple curves have this property, hence the proof is complete.

\noindent{\bf (ii)}
Put $\mu_0:=0$ and $\mu_i:=\sum_{j=1}^i\lambda_j$ for $i=1,\ldots,n$ and notice that from
\[
\sfd_\Y(\gamma_{t+\eps},\gamma_t)\leq\sum_{i=1}^n\sfd_\Y(\gamma_{t+\mu_{i}\eps},\gamma_{t+\mu_{i-1}\eps})
\]
and the triangle inequality in $L^p$ we obtain
\[
\|\sfd_\Y(\gamma_\cdot,\gamma_{\cdot+\eps})\|_{L^p([0,T-\eps])}\leq\sum_{i=1}^n\|\sfd_\Y(\gamma_{\cdot+\mu_i\eps},\gamma_{\cdot+\mu_{i-1}\eps})\|_{L^p([0,T-\eps])}\leq\sum_{i=1}^n\|\sfd_\Y(\gamma_{\cdot+\lambda_i\eps},\gamma_{\cdot})\|_{L^p([0,T-\lambda_i\eps])}.
\]
Then \eqref{eq:subpar} follows from the identity $(\mathcal E_{p,\eps}(\gamma))^{1/p}=\eps^{-1}\|\sfd_\Y(\gamma_\cdot,\gamma_{\cdot+\eps})\|_{L^p([0,T-\eps])}$.
\end{proof}

\begin{lemma}\label{le:contin0}
Let $T>0$ and $f:(0,T)\to\R^+$ be such that
\begin{equation}
\label{eq:subpf}
\lim_{\eps\downarrow0}\eps f(\eps)=0\qquad\textrm{ and }\qquad f(\eps)\leq\sum_{i=1}^n\lambda_if(\lambda_i\eps)
\end{equation}
for every $\eps\in(0,T)$, $n\in\N$ and  $\lambda_i\in[0,1]$, $i=1,\ldots,n$, $n\in\N$, with $\sum_i\lambda_i=1$.

Then there exists the limit as $\eps\downarrow0$ of $f(\eps)$ and
\begin{equation}
\label{eq:exlim}
f(\eps)\leq \lim_{\eps'\downarrow0}f(\eps')\qquad\forall \eps\in(0,T).
\end{equation}
\end{lemma}
\begin{proof}
Let $\eps',\eps\in(0,T)$ be with $\eps'\leq \eps$ and denote by $[x]$ the integer part of $x\in\R$. Apply the second in \eqref{eq:subpf} with $n:=[\frac{\eps}{\eps'}]+1$, $\lambda_i:=\frac{\eps'}{\eps}$ for every $i\leq n-1$ and $\lambda_n=1-[\frac{\eps}{\eps'}]\frac{\eps'}{\eps}$ to get
\[
f(\eps)\leq \Big[\frac{\eps}{\eps'}\Big]\frac{\eps'}{\eps}\,f(\eps')+\frac1\eps\Big(\eps-\eps'\Big[\frac{\eps}{\eps'}\Big] \Big) f\Big(\eps-\eps'\Big[\frac{\eps}{\eps'}\Big]\Big).
\]
Fix $\eps$ and notice that $[\frac{\eps}{\eps'}]\frac{\eps'}{\eps}\to1$ and $\eps-\eps'[\frac{\eps}{\eps'}]\to0 $ as $\eps'\downarrow0$, thus letting $\eps'\downarrow0$ in the above and taking into account the first in \eqref{eq:subpf} we deduce
\begin{equation}
\label{eq:forlim2}
f(\eps)\leq\limi_{\eps'\downarrow0} f(\eps').
\end{equation}
In particular $\lims_{\eps\downarrow0}f(\eps)\leq\limi_{\eps'\downarrow0} f(\eps')$, showing the existence of the limit; then \eqref{eq:exlim} follows from \eqref{eq:forlim2}.
\end{proof}
\begin{corollary}\label{cor:energy0}
Let $\gamma:[0,T]\to \Y$ be Borel. Then there exists the limit of $\mathcal E_{p,\eps}(\gamma)$ as $\eps\downarrow 0$ and
\begin{equation}
\label{eq:monotener}
 \mathcal E_p(\gamma)=\sup_{\eps\in(0,T)}\mathcal E_{p,\eps}(\gamma).
\end{equation}
\end{corollary}
\begin{proof}
Thanks to  Lemma \ref{le:baseep}  we can apply Lemma \ref{le:contin0} to the function $f(\eps):= (\mathcal E_{p,\eps}(\gamma))^{1/p}$. The conclusion follows.
\end{proof}
As in the real-valued case, there is a tight connection between the notions of absolutely continuous and Sobolev curves:
\begin{theorem}[Sobolev and AC curves]\label{thm:sobaccurve}
Let $(\Y,\sfd_\Y)$ be a complete space and $p\in(1,\infty)$. Then:
\begin{itemize}
\item[i)] Let $\gamma\in AC^p([0,T],\Y)$. Then (the equivalence class up to a.e.\ equality of) $\gamma$ belongs to $W^{1,p}([0,T],\Y)$.
\item[ii)] Let $\gamma\in W^{1,p}([0,T],\Y)$. Then
\begin{itemize}
\item[ii-a)]  $\gamma$ admits a continuous representative and such representative belongs to $AC^p([0,T],\Y)$.
\item[ii-b)] The functions $t\mapsto \frac{\sfd_\Y(\gamma_{t+h},\gamma_t)}{h}$ (set equal to 0 if $t+h\notin [0,T]$) have a strong limit $|\partial_t\gamma|$ - called distributional derivative of $\gamma$ - in $L^p([0,T])$ as $h\to 0$. In particular it holds
\begin{equation}
\label{eq:derint}
 \mathcal E_p(\gamma)=\int_0^T|\partial_t\gamma|^p\,\d t.
\end{equation}
\item[ii-c)] $|\partial_t\gamma|$ coincides for a.e.\ $t$ with the metric speed of the continuous representative of $\gamma$
\end{itemize}
\end{itemize}
Finally, $W^{1,p}([0,T],\Y)$ is a Borel subset of $L^p([0,T],\Y)$ and equipping it with the induced Borel structure we have that the map from $W^{1,p}([0,T],\Y)$ to $C([0,T],\Y)$ (resp.\ $L^p((0,T))$) sending a curve to its continuous representative (resp.\ distributional derivative) is Borel.
\end{theorem}
\begin{proof}\ \\
\noindent{\bf (i)} Being continuous, $\gamma$ is Borel and $\gamma([0,T])\subset \Y$ is compact, hence separable. It follows that $\gamma\in L^p([0,T],\Y)$. Given $h > 0$, the bound $\sfd(\gamma_{t+h},\gamma_t)\leq\int_t^{t+h}|\dot\gamma_s|\,\d s$ gives
\begin{equation}
\label{eq:oneside}
\int_0^{T-h}\frac{\sfd_\Y^p(\gamma_{t+h},\gamma_t)}{|h|^p}\,\d t\leq\int_0^{T-h}\Big|\frac1h\int_t^{t+h}|\dot\gamma_s|\,\d s\Big|^p\,\d t\leq \int_0^{T-h}\frac1h\int_t^{t+h}|\dot\gamma_s|^p\,\d s \,\d t\leq \int_0^T|\dot\gamma_s|^p\,\d s
\end{equation}
and given that this holds for every $h\in(0,T)$, the claim is proved.

\noindent{\bf (ii)}  Let $h_k\downarrow0$ be such that the functions $\frac{\sfd_\Y(\gamma_{t+h_k},\gamma_t)}{h_k}$ defined to be 0 if $t+h_k\notin[0,T]$ (which are uniformly bounded in $L^p((0,T))$ by assumption) weakly converge to some limit function $g$. Also, let  $\{x_n:n\in\N\}\subset\Y$ be countable and dense set and define $f_n(t):=\sfd_\Y(\gamma_t,x_n)$. Then the triangle inequality ensures  that $f_n\in L^p((0,T))$ and $|f_n(t+h)-f_n(t)|\leq\sfd_\Y(\gamma_{t+h},\gamma_t)$ and hence up to pass to a subsequence we can assume that for every $n\in\N$ the functions $t\mapsto \frac{f_n(t+h_k)-f_n(t)}{h_k}$ converge to some limit $g_n$ in the weak convergence of $L^p((0,T))$ as $h_k\downarrow0$. The construction ensures that $|g_n|\leq g$ and  for any  $\varphi\in C^\infty_c((0,T))$ we have
\[
-\int_0^Tf_n(t)\varphi'(t)\,\d t=\lim_{h_k\downarrow 0}\int\frac{f_{n}(t+h_k)-f_{n}(t)}{h_k}\varphi(t)\,\d t=\int_0^T g_n(t)\varphi(t)\,\d t,
\]
which shows that $f_n\in W^{1,p}((0,T))$ with $\partial_tf_n=g_n$.  It turn, it is well known - and easy to prove - that this implies that $f_n$ admits a continuous representative $\tilde f_n$ satisfying
\[
|\tilde f_n(s)-\tilde f_n(t)|\leq\int_t^s|\partial_r f_n|\,\d r\leq\int_t^s g(r)\,\d r,\qquad\forall t,s\in[0,T].
\]
Thus for  $\mathcal N\subset [0,T]$ Borel, negligible and such that $f_n(t)=\tilde f_n(t)$ for any $t\notin\mathcal N$ and $n\in\N$ we have
\[
\sfd_\Y(\gamma_t,\gamma_s)=\sup_n|f_n(t)-f_n(s)|\leq \int_t^s g(r)\,\d r\qquad\forall t,s\in [0,T]\setminus \mathcal N,
\]
which, e.g.\ by the absolute continuity of the integral, shows that the restriction of $\gamma$ to $[0,T]\setminus \mathcal N$ is uniformly continuous and thus it can be extended to a continuous curve $\tilde\gamma$ which clearly satisfies
\[
\sfd_\Y(\tilde\gamma_t,\tilde\gamma_s) \leq \int_t^s g(r)\,\d r\qquad\forall t,s\in [0,T].
\]
By the very definition, this means $\tilde\gamma\in AC^p([0,T],\Y)$ and Theorem \ref{thm:ms} also tells that
\begin{equation}
\label{eq:secondside}
|\dot{\tilde \gamma}_t|\leq g(t)\qquad \text{a.e.}\ t,
\end{equation}
therefore the chain of inequalities
\[
\begin{split}
\|g\|_{p}^p\leq\limi_{h_k\downarrow0}\Big\|\frac{\sfd_\Y(\gamma_{\cdot+h_k},\gamma_\cdot)}{h_k}\Big\|_p^p\leq\lims_{h_k\downarrow0}\Big\|\frac{\sfd_\Y(\gamma_{\cdot+h_k},\gamma_\cdot)}{h_k}\Big\|_p^p\stackrel{\eqref{eq:oneside}}\leq \||\dot{\tilde \gamma}|\|_p^p\stackrel{\eqref{eq:secondside}}\leq \|g\|_{p}^p
\end{split}
\]
is actually made of equalities. In particular, the functions $\frac{\sfd_\Y(\gamma_{\cdot+h_k},\gamma_\cdot)}{h_k}$ converge to $g$ also in norm, thus strongly in $L^p((0,T))$
(because $L_{p}$-spaces are uniformly convex when $p \in (1,\infty)$). Also, the equality $ \||\dot{\tilde \gamma}|\|_p^p= \|g\|_{p}^p$ and \eqref{eq:secondside} force $g=|\dot{\tilde \gamma}|$ showing at once that the limit of $\frac{\sfd_\Y(\gamma_{\cdot+h},\gamma_\cdot)}{h}$ as $h\downarrow0$ does not depend on the particular subsequence chosen and that it coincides with the metric speed of the continuous representative of $\gamma$.

\noindent{\bf Final statements} It is clear that the functionals $\mathcal E_{p,\eps}:L^p([0,T],\Y)\to[0,+\infty]$ are continuous, thus \eqref{eq:monotener} grants that $\mathcal E_p$ is lower semicontinuous. Hence it is  Borel and since $W^{1,p}([0,T],\Y)=\{\gamma:\mathcal E_p(\gamma)<\infty\}$ we see that the set of all Sobolev curves is a Borel subset
of $L_{p}([0,T],\Y)$ as well.

For the Borel regularity of the  'continuous representative' map it is sufficient to show that for any $c>0$ such map is continuous from $\{\mathcal E_p\leq c\}\subset L^p([0,T],\Y)$ to $C([0,T],\Y)$. Thus let $\gamma^n\to\gamma$ in $L^p([0,T],\Y)$ be with $\sup_n\mathcal E_p(\gamma_n)\leq c$. As for the classical $L^p$ spaces - and with the same proof - up to pass to a subsequence we can assume that $\gamma^n_t\to \gamma_t$ for a.e.\ $t$. Now notice that the bound
\[
\sfd_\Y(\gamma^n_t,\gamma^n_s)\leq \int_t^s|\partial_r\gamma^n_r|\,\d r \leq \Big|\int_t^s|\partial_r\gamma^n_r|^p\,\d r\Big|^{\frac1p} \,\Big|\int_t^s1\,\d r\Big|^{1-\frac1p}\leq c|s-t|^{1-\frac1p}
\]
grants that the curves $\gamma^n$ are uniformly continuous, thus from pointwise a.e.\ convergence we deduce uniform convergence to a limit curve $\tilde\gamma$. It is then  clear that $\tilde\gamma$ is the continuous representative of $\gamma$, thus concluding the proof of the claim.

Finally, the Borel regularity of the `distributional derivative' map follows easily noticing that for any $n\in\N$ the map $\gamma\mapsto \frac{\sfd_\Y(\gamma_{t+1/n},\gamma_t)}{1/n}$ from $L^p([0,T],\Y)$ to $L^p((0,T))$  (set 0 if $t+1/n\notin[0,T]$) is continuous, hence Borel. The conclusion follows noticing that $W^{1,p}([0,T],\Y)$ coincides with the class of $\gamma$'s such that the maps have limit in $L^p((0,T))$ as $n\to\infty$, the distributional derivative being such limit.
\end{proof}
A direct and simple corollary of the above is the following chain rule:
\begin{corollary}\label{cor:chain}
Let $p\in(1,\infty)$, $\gamma\in W^{1,p}([0,1],\X)$ (resp. $AC^p([0,1],\X)$) and $\varphi:\X\to\Y$ Lipschitz. Then $\varphi\circ\gamma\in W^{1,p}([0,1],\Y)$ (resp. $AC^p([0,1],\Y)$) with distributional derivative (resp. metric speed) bounded from above by $\lip\varphi\circ\gamma |\partial_t\gamma|$.
\end{corollary}
\begin{proof}
Assume that $\gamma\in AC^p([0,1],\X)$. Then directly from the definition it is clear that $\varphi\circ\gamma\in AC^p([0,1],\Y)$, while from the definition of metric speed it follows the desired bound on the metric speed of the composition. The case of Sobolev curves now follows from Theorem \ref{thm:sobaccurve}.
\end{proof}

For later use let us point out the following simple lemma:
\begin{lemma}\label{le:spcount}
Let $(\X,\sfd)$ be a complete and separable space. Then there exists a countable family $(f_n)$ of 1-Lipschitz functions with bounded support such that
\begin{equation}
\label{eq:distsup}
\sfd(x,y)=\sup_n(f_n(x)-f_n(y))\qquad\forall x,y\in\X
\end{equation}
and any such family has the following property:

For any $p\in(1,\infty)$ and $\gamma\in L^p([0,1],\X)$ we have that $\gamma\in  W^{1,p}([0,1],\X)$ if and only if $f_n\circ\gamma\in W^{1,p}((0,1))$ for every $n\in\N$ with $\sup_n\partial_t(f_n\circ\gamma)\in L^p((0,1))$. Moreover, if these holds we also have
\[
|\partial_t\gamma_t|=\sup_n\partial_t(f_{n}\circ\gamma)_t.
\]
\end{lemma}
\begin{proof} To prove that a countable family for which \eqref{eq:distsup} holds exists, simply pick $f_{k,n}(x):=(k-\sfd(x,x_n))\vee 0$, where  $(x_n)\subset\X$ is countable and dense.

For the second part of the claim, let $(f_n)$ be an arbitrary sequence of 1-Lipschitz functions for which \eqref{eq:distsup} holds and notice that for every absolutely continuous curve $\gamma$ the function $f_n\circ\gamma$ is absolutely continuous with $(f_n\circ\gamma)'_t\leq |\dot\gamma_t|$. Also, for any $t,s\in[0,1]$, $t\leq s$ we have
\[
\sfd(\gamma_t,\gamma_s)=\sup_nf_n(s)-f_n(t)=\sup_n\int_t^s (f_{n}\circ\gamma)'_r\,\d r\leq\int_t^s\sup_n (f_{n}\circ\gamma)'_r\,\d r,
\]
showing that $\sup_n (f_{n}\circ\gamma)'_t\geq|\dot\gamma_t|$ for a.e.\ $t$. Therefore Theorem \ref{thm:sobaccurve} ensures that if $\gamma\in W^{1,p}([0,1],\X)$, then the claimed properties of $f_n\circ\gamma$ all hold.

Conversely, assume that $\gamma\in L^p([0,1],\X)$ is such that $f_n\circ\gamma\in W^{1,p}((0,1))$ for every $n\in\N$ with $\sup_n\partial_t(f_n\circ\gamma)\in L^p((0,1))$. Then the same arguments used in the proof of point $(ii)$ of Theorem \ref{thm:sobaccurve} give that $\gamma$ admits a continuous representative in $AC^p([0,1],\X)$ and the conclusion follows.
\end{proof}
We now study the particular case of  curves with values in some $L^p$ space.
\begin{lemma}\label{le:basequ} Let $(\X,\sfd,\mm)$ be a complete and separable metric space equipped with a non-negative and non-zero Radon measure finite on bounded sets. Equip $\X\times[0,1]$ with the product of $\mm$ and the Lebesgue measure and  let $(t,x)\mapsto f_t(x)\in \R$ be a given Borel map in $L^p(\X\times[0,1])$, for $p\in(1,\infty)$.

Then the following are equivalent:
\begin{itemize}
\item[i)] The map $[0,1]\ni t\mapsto f_t\in L^p(\X)$ belongs to $W^{1,p}([0,1],L^p(\X))$ (resp. $AC^p([0,1],L^p(\X))$).
\item[ii)] There is a function $g\in L^p(\X\times[0,1])$, $g\geq 0$, such that for a.e.\ (resp.\ every) $t,s\in[0,1]$, $t<s$ it holds
\begin{equation}
\label{eq:geq}
|f_s-f_t|\leq \int_t^s g_r\,\d r\quad\mm-a.e..
\end{equation}
\item[iii)] For $\mm$-a.e.\ $x\in\X$ we have $f_\cdot(x)\in W^{1,p}([0,1])$ and the function $(t,x)\mapsto |\partial_tf_t(x)|=:h_t(x)$ belongs to $L^p(\X\times[0,1])$ (resp.\ and moreover $(f_t)\in C([0,1],L^p(\X))$).
\end{itemize}
Moreover, if these holds $h$ is the least function $g\geq 0$, in the $\mm\times\mathcal L^1$-a.e.\ sense, for which \eqref{eq:geq} holds and
\begin{equation}
\label{eq:splpu}
\frac{|f_{t+\eps}(x)-f_t(x)|}{|\eps|}\quad\to\quad h_t(x)\qquad\text{ in $L^p(\X\times[0,1])$ as $\eps\to 0$,}
\end{equation}
where the incremental ratios are defined to be 0 if $t+\eps\notin[0,1]$.
\end{lemma}
\begin{proof}\  We shall deal with the absolutely continuous case, as the Sobolev one can be obtained through very similar arguments taking also into account Theorem \ref{thm:sobaccurve}.

\noindent{$\mathbf{ (i)\Rightarrow(ii)}$} Recall that $L^p(\X)$ has the Radon-Nikodym property and let $(\tilde g_t)\in L^p([0,1],L^p(\X))$ be the derivative of   $(f_t)$. Then for every  $t,s\in[0,1]$, $t<s$ it holds
\[
f_s-f_t=\int_t^s \tilde g_r\,\d r,
\]
the integral being intended in the Bochner sense. By classical arguments (see e.g.\ \cite[Proposition 1.3.19]{GP19}) we deduce the existence
of a Borel function $g \in L^{p}(\X\times[0,1])$ such that for $\mathcal{L}^{1}$-a.e. $t \in [0,1]$ we have $g_{t}(x)=|\tilde g_t(x)|$ for $\mm-a.e.$ $x \in \X$.
It is then easy to see that the same formula holds also $\mm$-a.e., so that the conclusion holds with $g_t:=|\tilde g_t|$.

\noindent{$\mathbf{ (ii)\Rightarrow(i)}$}  For any $t,s\in[0,1]$, $t<s$ from  \eqref{eq:geq}  we have
\[
\|f_s-f_t\|_{L^p(\X)}\leq \int_t^s \|g_r\|_{L^p(\X)}\,\d r
\]
and since the identity $\int_0^1\|g_t\|_{L^p(\X)}^p\,\d t=\iint_0^1|g_t|^p\,\d t\,\d\mm<\infty$ shows that $(\|g_t\|_{L^p(\X)})\in L^p((0,1))$, this is sufficient to conclude.

\noindent{$\mathbf{ (ii)\Rightarrow(iii)}$} Continuity follows from the already proved implication $(ii)\Rightarrow(i)$. The assumption is equivalent to asking that for a.e.\ $t,\eps\in[0,1]$ with $t+\eps\in[0,1]$ it holds
\begin{equation}
\label{eq:geqvar}
|f_{t+\eps}-f_t|\leq \int_t^{t+\eps} g_r\,\d r\quad\mm-a.e..
\end{equation}
Notice also that  Fubini's theorem and the assumption $g\in L^p(\X\times [0,1])$ give that $t\mapsto g_t(x)\in L^p((0,1))$ for $\mm$-a.e.\ $x$.

Now observe that for given $\tilde f\in L^1((0,1))$, $\tilde g\in L^p((0,1))$ we have $\tilde f\in W^{1,p}((0,1))$ with $|\partial_t\tilde f_t|\leq \tilde g_t$ a.e.\ $t$ if and only if for every $\varphi\in C^1_c((0,1))$, $\varphi\geq 0$ it holds
\[
\Big|\int_0^1 \tilde f_t\varphi'_t\,\d t\Big|\leq\int_0^1\tilde g_t\varphi_t\,\d t.
\]
Also, by a simple approximation argument  it is sufficient to check the above for $\varphi$ running in a countable set $D$, dense in the $C^1$-topology in the class of admissible $\varphi$'s.

With this said, for any $\varphi\in D$ we have
\[
\begin{split}
\int_0^1 f_t\varphi'_t\,\d t=\lim_{\eps\downarrow0}\int_0^1 f_t\frac{\varphi_{t+\eps}-\varphi_t}\eps \,\d t=-\lim_{\eps\downarrow0}\int_0^1\varphi_t\frac{f_t-f_{t-\eps}}{\eps}\,\d t\qquad\mm-a.e.
\end{split}
\]
and therefore \eqref{eq:geqvar} gives that $\mm$-a.e.\ the bound
\[
\Big|\int_0^1 f_t\varphi'_t\,\d t\Big|=\lim_{\eps\downarrow0}\Big|\int_0^1\varphi_t\frac{f_t-f_{t-\eps}}{\eps}\,\d t\Big|\leq\lim_{\eps\downarrow0}\int_0^1\varphi_t\frac1\eps\int_{t}^{t+\eps}g_r\,\d r\,\d t=\int_0^1\varphi_tg_t\,\d t
\]
holds for every $\varphi\in D$. According to what previously said, this is sufficient to conclude.

\noindent{$\mathbf{ (iii)\Rightarrow(ii)}$}  For $\mm$-a.e.\ $x$ we know that: for a.e.\ $t,s\in[0,1]$ with $t<s$ it holds
\[
|f_s(x)-f_t(x)|=\big|\int_t^s\partial_rf_r(x)\,\d r\big|\leq\int_t^s|\partial_rf_r(x)|\,\d r
\]
and thus Fubini's theorem gives that for a.e.\ $t,s\in[0,1]$ with $t<s$ it holds
\[
|f_s-f_t|\leq \int_t^s|\partial_rf_r|\,\d r\qquad\mm-a.e..
\]
To conclude that the same holds for every $t,s$, notice that the continuity assumption on $(f_t)$ grants that the left hand side of the above is continuous in $t,s$ with values in $L^p(\X)$. The same holds for the right hand side due to the assumption $\partial_tf_t\in L^p(\X\times[0,1])$.

\noindent{\bf{Last statements}} The fact that $h$ is the least $g\geq 0$ for which \eqref{eq:geq} holds follows by the arguments given. To prove \eqref{eq:splpu} notice that by standard results about $W^{1,p}((0,1))$ functions we know that for $\mm$-a.e.\ $x\in\X$ the given incremental ratios converge to $t\mapsto h_t(x)$ in $L^p((0,1))$. Hence by dominate convergence and Fubini's theorem the conclusion follows if we show that the incremental ratios are dominated in $L^p(\X\times[0,1])$. This is a direct consequence of \eqref{eq:geq}.
\end{proof}

\subsection{Regular Lagrangian Flows}
Here we very briefly recall the main definitions and results of the metric theory of Regular Lagrangian Flows as developed in \cite{Ambrosio-Trevisan14}. The concept of Regular Lagrangian Flow provides the correct substitute, in this setting, for the concept of solution of the ODE
\[
\gamma_t'=Z_t(\gamma_t),
\]
see also \cite{Ambrosio2008} and \cite{AT15} for  overviews of the subject and historical remarks on $\R^d$ and $\RCD$ spaces respectively.

\bigskip

We begin with:

\begin{definition}[Regular Lagrangian Flow]
Let  $(\X,\sfd,\mm)$ be a complete and separable metric space equipped with a non-negative and non-zero Radon measure giving finite mass to bounded sets.

Let $(Z_t)\subset L^1([0,1],L^2(T\X))$ be a given family of vector fields. A Regular Lagrangian Flow of $(Z_t)$ is a Borel map $\Fl^Z:[0,1]\times \X\to \X$ such that:
\begin{itemize}
\item[i)] For some $C>0$ it holds
\begin{equation}
\label{eq:bdcompr}
(\Fl^Z_t)_*\mm\leq C\mm\qquad\forall t\in[0,1].
\end{equation}
\item[ii)] For every $f\in W^{1,2}(\X)$ it holds: for $\mm$-a.e.\ $x$  the map $t\mapsto f(\Fl^Z_t(x))$ is in $W^{1,1}(0,1)$  with
\[
\partial_tf(\Fl^Z_t(x))=\d f(Z_t)\big( \Fl^Z_t(x)\big)\qquad a.e.\ t\in[0,1].
\]
\end{itemize}
\end{definition}
Notice that part of the role of $(i)$ is to ensure that $(ii)$ makes sense, as the function $\d f(Z_t)$ is only defined up to $\mm$-a.e.\ equality, so that its composition with $\Fl^Z_t$ makes sense because $(\Fl^Z_t)_*\mm\ll\mm$. We also point out that in \cite{Ambrosio-Trevisan14} property $(ii)$ is only required for a certain family of functions, labelled $\mathcal A$, dense in $W^{1,2}(\X)$. It is easy to see that this is the same as imposing the same property for any $f\in W^{1,2}(\X)$ because given such $f$ and $(f_n)\subset\mathcal A$ converging to $f$ in the $W^{1,2}$-topology, we have that $\d f_n\to \d f$ in $L^2(T^*\X)$. Therefore taking into account the integrability assumption on $(Z_t)$ and \eqref{eq:bdcompr} we deduce that the sequence of maps $(t,x)\mapsto \d f_n(Z_t)(\Fl^Z_t(x))$ converge to $(t,x)\mapsto \d f(Z_t)(\Fl^Z_t(x))$ in $L^1_t(L^1_x)\sim L^1_x(L^1_t)$. Similarly, $(t,x)\mapsto f_n(\Fl^Z_t(x))$ converge to $(t,x)\mapsto f(\Fl^Z_t(x))$ in $L^2_t(L^2_x)\sim L^2_x(L^2_t)$. Hence up to pass to a not-relabeled subsequence we deduce that for $\mm$-a.e.\ $x$ the functions $t\mapsto f_n(\Fl^Z_t(x)),\,\d f_n(Z_t)(\Fl^Z_t(x))$ converge to $t\mapsto f(\Fl^Z_t(x)),\,\d f(Z_t)(\Fl^Z_t(x))$ in the $L^1$ and $L^2$ topologies respectively. This is sufficient to pass to the limit in property $(ii)$ for the $f_n$'s and obtain that the same property holds for $f$.

The concept of Regular Lagrangian Flow is tightly linked to the continuity equation:
\begin{definition}[Continuity equation]Let  $(\X,\sfd,\mm)$ be a complete and separable metric space equipped with a non-negative and non-zero Radon measure giving finite mass to bounded sets.

Let $T>0$, $[0,T]\ni t\mapsto \rho_t\in L^\infty(\X)$ be a weakly$^*$ continuous curve of probability densities and $[0,T]\ni t\mapsto v_t\subset L^0(T\X)$ be a Borel map. We say that $(\rho_t,v_t)$ solves the continuity equation provided:
\begin{itemize}
\item[i)] We have $\displaystyle\iint_0^T|v_t|^2\rho_t\,\d t\,\d\mm<\infty$
\item[ii)] For any $f\in W^{1,2}(\X)$ the map $t\mapsto\int f\rho_t\,\d\mm$ is absolutely continuous and its derivative is given for $\mathcal{L}^{1}$-a.e. $t \in [0,T]$ by
\[
\frac{\d}{\d t}\int f\rho_t\,\d\mm=\int \d f(v_t)\rho_t\,\d\mm.
\]
\end{itemize}
\end{definition}
Albeit the last two definitions make sense on arbitrary metric measure spaces, to develop a good theory it seems necessary to impose a lower bound on the Ricci curvature. In particular, the following notion of regularity for vector fields is important:
\begin{definition}[Regular vector fields]
We say that a vector field $Z$ over a $\RCD(K,\infty)$ space $\X$ is regular provided $Z\in L^\infty \cap W^{1,2}_{C}(T\X)$ and moreover it  is in the domain of the divergence with $(\div Z)^-\in L^\infty(\X)$. For a Borel time-dependent vector field $(Z_t)$ defined for $t\in[0,T]$ we say that it is regular provided  $Z_t$ is regular in the previous sense for a.e.\ $t$ and \[
\int_0^T\||Z_t|\|_{L^\infty}+\||\nabla Z_t|_{\sf HS}\|_{L^2}+\|(\div Z_t)^-\|_{L^\infty}\,\d t<\infty.
\]
\end{definition}
The main/basic result of the theory of Regular Lagrangian Flows on $\RCD$ spaces is:
\begin{theorem}\label{thm:AT}
Let $(\X,\sfd,\mm)$ be a $\RCD(K,\infty)$ space and $(Z_t)$ a regular vector field parametrized on $t\in I\subset \R$. Then:
\begin{itemize}
\item[i)] There is a unique Regular Lagrangian flow $(\Fl^{Z}_t)$ of $(Z_t)$ (uniqueness is intended at the level of curves, i.e.: if $(\tilde \Fl_t)$ is another flow, then for $\mm$-a.e.\ $x$ it holds $\Fl^{Z}_t(x)=\tilde\Fl_t(x)$ for any $t\in I$).
\item[ii)] For any bounded probability density $\rho_0$ with bounded support there is a unique family $(\rho_t)$, $t\in I$, such that $\rho_t\leq C$ for some $C>0$ and every $t\in I$ and $(\rho_t,Z_t)$ solves the continuity equation. Moreover,  $\rho_t$ is the density w.r.t.\ $\mm$ of $(\Fl^{Z}_t)_*(\rho_0\mm)$ and the Regular Lagrangian Flow is the only flow with this property.
\item[iii)] For $\mm$-a.e.\ $x$ the curve $t\mapsto \Fl^Z_t(x)$ is absolutely continuous and its  metric speed ${\rm ms}(\Fl^Z_\cdot(x),t)$ at time $t$ is given by
\begin{equation}
\label{eq:ms}
{\rm ms}(\Fl^Z_\cdot(x),t)=|Z_t|(\Fl^Z_t(x))\qquad\mm-a.e.\ x\ a.e.\ t.
\end{equation}
Also,  for $(\rho_t)$ as above we have
\begin{equation}
\label{eq:bcompr}
\|\rho_s\|_{L^\infty}\leq \|\rho_t\|_{L^\infty}\exp\Big(\int_t^s\|(\div Z_r)^-\|_{L^\infty}\,\d r\Big)\qquad\forall t,s\in I,\ t\leq s.
\end{equation}
\end{itemize}
\end{theorem}

\begin{proof}\ \\

\noindent{\bf (i)} follows from (ii) and Theorem 8.3 in \cite{Ambrosio-Trevisan14}.

\noindent{\bf (ii)} for the existence see for instance Theorem 6.1 in \cite{AT15} and notice that Example 4.1 ensures that in our setting such theorem is applicable. For uniqueness see Theorem 6.4 in \cite{AT15} (and recall that Corollary 6.3 in \cite{Ambrosio-Trevisan14} grants that a $L^4-\Gamma$ estimate holds).

\noindent{\bf (iii)} the fact that the flow is concentrated on a family of absolutely continuous curves and inequality $\leq$ in \eqref{eq:ms} follows from the superposition principle, Definition 7.3 in \cite{Ambrosio-Trevisan14} and our Lemma \ref{le:basequ}. The opposite inequality follows from Lemma \ref{le:spcount} and the definition of norm of a vector field (equivalently, notice that for every probability measure $\mu\leq C\mm$ on $X$ the plan $\pi:=( \Fl_\cdot)_*\mu$ is a test plan, that the superposition principle tells that $\pi'_t= e_t^*Z_t$ and conclude with Theorem 2.3.18 in \cite{Gigli14}).

Finally, the estimate \eqref{eq:bcompr} comes from estimate (4-15) in \cite{Ambrosio-Trevisan14}.
\end{proof}

\section{New stability results for Regular Lagrangian flows}

\subsection{The $\RCD$ setting}\label{se:rcdrlf}
Aim of this section is to prove new stability results for Regular Lagrangian Flows on $\RCD$ spaces. The main application that we have in mind is a sort of Trotter-Kato formula for Regular Lagrangian Flows, namely  that if $Z_1,Z_2$ are regular vector fields, then for every $t\geq 0$ the maps
\[
\underbrace{(\Fl^{Z_1}_{\frac t n}\circ\Fl^{Z_2}_{\frac tn})\circ\cdots\circ(\Fl^{Z_1}_{\frac tn}\circ\Fl^{Z_2}_{\frac tn})}_{n\ \text{times}}
\]
converge to $\Fl^{Z_1+Z_2}_t$ as $n\to\infty$, see Theorem \ref{thm:stabfl} and Proposition \ref{prop:sumvect} for the precise statement.

\bigskip

We start with the following very general fact:
\begin{lemma}\label{le:convmeas}
Let $(\X,\sfd,\mm)$ be a metric measure space, $(\Y,\sfd_\Y)$ a complete and separable metric space and $T_n:\X\to \Y$, $n\in\N\cup\{\infty\}$, be such that: for any bounded probability density $\rho$ with bounded support  the sequence $n\mapsto(T_n)_*(\rho\mm)$ weakly converges to $(T_\infty)_*(\rho\mm)$ in duality with $C_b(\Y)$.

Then $T_n\to T_\infty$ locally in measure.
\end{lemma}
\begin{proof}
Suppose not. Then there are $\eps >0$ and $E\subset\X$ bounded such that
\begin{equation}
\label{eq:man}
\text{for $A_n:=\{x\in E:\sfd_\Y(T_n(x),T_\infty(x))\geq \eps\}$ it holds}\quad\mm(A_n)\geq \eps\qquad\text{for infinitely many $n$'s}.
\end{equation}
Now consider the measure $\nu:=(T_\infty)_*(\mm\restr E)$ on $\Y$ which, being Borel and finite, is Radon. Thus there is a compact $K\subset \Y$ such that $\nu(Y\setminus K)\leq \frac\eps2$ and by compactness of $K$ we can find a finite number of sets $B_1,\ldots,B_k$ of diameter $\leq \frac\eps2$ covering $K$. Define $C_i:=T_\infty^{-1}(B _i)$, $i=1,\ldots,k$, so that by definition $\mm(E\setminus\cup_iC_i)\leq\frac\eps2$ and thus by \eqref{eq:man} we deduce that  $\mm(A_n\cap\cup_iC_i)\geq\frac\eps2$ for an infinite number of $n$'s. Hence
\begin{equation}
\label{eq:man2}
\text{for some $\bar i\in\{1,\ldots,k\}$ we have $\mm(A_n\cap C_{\bar i})\geq\frac\eps{2k}$ for an infinite number of $n$'s.}
\end{equation}
Now let $\varphi\in C_b(\Y)$ be with values in $[0,1]$ identically 1 on $B_{\bar i}$ and with support in the $\frac\eps2$-neighbourhood of $B_{\bar i}$, so that by construction $\varphi\circ T_\infty\restr{C_i}\equiv1$ and
\begin{equation}
\label{eq:man3}
\text{$\varphi\circ T_n\restr{A_n\cap C_i}\equiv0$ for every $n$ such that \eqref{eq:man} holds}
\end{equation}
(because if $T_\infty(x)\in B_{\bar i}$ and $\sfd_\Y(T_n(x),T_\infty(x))\geq \eps$ then since the diameter of  $B_{\bar i}$ is $\leq \frac\eps2$ we necessarily have that $T_n(x)$ is not in the $\frac\eps2$-neighbourhood of $B_{\bar i}$).

Finally, put $\rho:=\nchi_{C_{\bar i}}\mm(C_{\bar i})^{-1}$ (notice that \eqref{eq:man2} ensures that $\mm(C_{\bar i})>0$ so that $\rho$ is well defined), observe that $\rho$ is a bounded probability density with bounded support and that putting $\mu:=\rho\mm$ by construction we have $\int\varphi\,\d (T_{\infty})_*\mu=1$ and
\[
\int\varphi\,\d(T_n)_*\mu=\int_{A_n}\varphi\circ T_n\,\d \mu+\int_{\X\setminus A_n}\varphi\circ T_n\,\d \mu\stackrel{\eqref{eq:man3}}\leq\mu(\X\setminus A_n)\leq 1-\mu(A_n\cap C_{\bar i})\stackrel{\eqref{eq:man2}}\leq 1-\frac\eps{2k}
\]
for every $n$ for which \eqref{eq:man2} holds. Hence $\limi_{n\to\infty}\int\varphi\,\d(T_n)_*\mu<\int\varphi\,\d(T_\infty)_*\mu$ violating the weak convergence of $((T_n)_*\mu)$ to $(T_\infty)_*\mu$ and thus concluding the proof.
\end{proof}
We can use such abstract result in conjunction with the theory of Regular Lagrangian Flows to deduce the following stability result, which links stability of solutions of the continuity equation to stability of the associated flows:
\begin{proposition}\label{le:convfl}
Let $(\X,\sfd,\mm)$ be a $\RCD(K,\infty)$ space and $(Z_n)$, $n\in\N\cup\{\infty\}$ regular time dependent vector fields such that $S_t:=\sup_{n\in\N\cup\{\infty\}}\||Z_{n,t}|\|_{L^\infty(\X)}\in L^1((0,1))$. Assume that for any probability density $\rho_0\in L^\infty(\X)$ with bounded support, letting  $(\rho_{n,t})$ be  the solution of the continuity equation for $Z_n$ starting from $\rho_0$, we have
\[
\rho_{n,t}\quad\weakto\quad\rho_{\infty,t}\qquad\forall t\in[0,1],
\]
weakly in duality with $C_b(\X)$.

Then $\Fl^{Z_n}\to\Fl^{Z_\infty}$ locally in measure as maps from $\X$ to $C([0,1],\X)$. In particular, for every $t\in[0,1]$ we have $\Fl_t^{Z_n}\to\Fl_t^{Z_\infty}$ locally in measure as maps from $\X$ to $\X$.
\end{proposition}
\begin{proof}
Recall from $(ii)$ of Theorem \ref{thm:AT}  that $\rho_{n,t}$ is the density w.r.t.\ $\mm$ of $(\Fl^{Z_n}_t)_*(\rho_0\mm)$. Hence our assumptions and Lemma \ref{le:convmeas} above grant that  $\Fl_t^{Z_n}\to\Fl_t^{Z_\infty}$ locally in measure as maps from $\X$ to $\X$ for every $t\in[0,1]$.

Now let $\mm'\in\pr(\X)$ be such that $\mm\ll\mm'\ll\mm$ and recall that the local convergence in measure of maps from $\X$ to $C([0,1],\X)$ is metrized by the distance
\[
\bar \sfd(\Fl,\Fl'):=\int1\wedge \sup_{t\in[0,1]}\sfd(\Fl_t(x),\Fl'_t(x))\,\d\mm'(x).
\]
Fix $\eps>0$ and let $k\in\N$ be such that $\int_{\frac ik}^{\frac{i+1}k}S_t\,\d t<\eps$ for every $i=0,\ldots,k-1$. Notice that   \eqref{eq:ms} grants that  $\sfd(\Fl^{Z_n}_t(x),\Fl^{Z_n}_s(x))\leq \int_t^s S_r\,\d r$ for $\mm$-a.e.\ $x\in \X$  and $n\in\N\cup\{\infty\}$ and therefore
\[
\sfd(\Fl_t(x),\Fl'_t(x))\leq \sfd(\Fl_{\frac ik}(x),\Fl'_{\frac ik}(x))+2\eps\qquad\text{for $i\in\{0,\ldots,k-1\}$ such that }t\in[\tfrac ik,\tfrac{i+1}k),
\]
for $\mm$-a.e.\ $x\in\X$, which gives
\[
\sup_{t\in[0,1]}\sfd(\Fl_t(x),\Fl'_t(x))\leq 2\eps+\sum_{i=0}^{k-1}\sfd(\Fl_{\frac ik}(x),\Fl'_{\frac ik}(x)),\qquad\mm-a.e.\ x.
\]
Hence we have
\[
\lims_{n\to\infty}\bar\sfd(\Fl^{Z_n},\Fl^{Z_\infty})\leq 2\eps+\lims_{n\to\infty}\sum_{i=0}^{k-1}\int 1\wedge \sfd(\Fl^{Z_n}_{\frac ik}(x),\Fl^{Z_{\infty}}_{\frac ik}(x))\,\d\mm'(x)=2\eps
\]
having used the local convergence in measure of $(\Fl^{Z_n}_{\frac ik})$ to $\Fl^{Z_\infty}_{\frac ik}$. The arbitrariness of $\eps$ gives the claim.
\end{proof}
Our question is now to find appropriate conditions on a sequence of vector fields which ensure convergence of solutions of the continuity equation. We shall work with:
\begin{definition}\label{def:wtss} Let $(Z_n)\subset L^1([0,1],L^1(T\X))$, $n\in\N\cup\{\infty\}$. We say that $Z_n\to Z_\infty$ weakly in time and strongly in space provided for any $\varphi\in C_c(\R)$ we have $Z_n^\varphi\to Z_\infty^\varphi$ strongly in $L^1([0,1],L^1(T\X))$, where  we put
\[
Z_{n,t}^\varphi:=\int_\R \varphi(t-s)Z_{n,s}\,\d s \qquad \forall t\in[0,1],\ n\in\N\cup\{\infty\},
\]
and it is intended that $Z_{n,s}=0$ for $s\notin[0,1]$.
\end{definition}
The main result of this section is the following theorem:
\begin{theorem}\label{thm:stabfl} Assume that $(Z_n)\subset L^1([0,1],L^1(T\X))$ converges weakly in time and strongly in space to the regular vector field $Z_\infty$. Assume also that \begin{equation}
\label{eq:unifZ}
\int_0^1\sup_{n\in \N\cup\{\infty\}}\Big(\||Z_{n,t}|\|_{L^\infty}+\||\nabla Z_{n,t}|_{\sf HS}\|_{L^2}\Big)\,\d t+\sup_{n\in \N\cup\{\infty\}}\int_0^1\|(\div Z_{n,t})^-\|_{L^\infty}\,\d t <\infty.
\end{equation}
Then the Regular Lagrangian Flows $(\Fl^{Z_n})$ converge locally in measure to the Regular Lagrangian Flow $\Fl^{Z_\infty}$ and, for every $t\in[0,1]$, the maps $(\Fl^{Z_n}_t)$ converge locally in measure to $\Fl^{Z_\infty}_t$.
\end{theorem}
\begin{proof}
Let $\rho_0\in L^\infty$ be a probability density with bounded support and let  $(\rho_{n,t})$ be the solution of the continuity equation for $Z_n$ starting from $\rho_0$. According to  Proposition \ref{le:convfl}, to conclude it is sufficient to prove that $\rho_{n,t}\weakto\rho_{\infty,t}$ in duality with $C_b(\X)$ for any $t\in[0,1]$.

{\bf Compactness} To this aim start observing that by the assumption \eqref{eq:unifZ} and the bound \eqref{eq:bcompr} we have that $\rho_{n,t}\leq C\mm$ for some $C>0$ independent on $n,t$. Analogously, from the bound \eqref{eq:ms} it easily follows that $\supp(\rho_{n,t})\subset B$ for some bounded closed set $B\subset \X$ independent on $n,t$.  Thus
\begin{equation}
\label{eq:boundunif}
\rho_{n,t}\mm\leq C\mm\restr B\qquad\forall n\in\N,\ t\in[0,1]
\end{equation}
and since $\mm\restr B$ is a finite Radon measure we can conclude that the family $\{\rho_{n,t}\}_{n,t}$ is tight. Finally, again the bound \eqref{eq:ms}  gives that the curves $t\mapsto \rho_{n,t}\mm$ are $W_1$-equiLipschitz. This and the previous observations imply that such sequence of curves is precompact in $C([0,1],(\mathscr P_1(\X),W_1))$ and thus up to pass to a non-relabeled subsequence we can assume that it converges to a limit curve  $(\mu_t)\in C([0,1],(\mathscr P_1(\X),W_1))$. Since the bound \eqref{eq:boundunif} passes to the limit, we have that $\mu_t=\eta_t\mm$ for some $\eta_t\leq C\nchi_B$ for every $t\in[0,1]$.

{\bf Identification of the limit}  To conclude it is sufficient to show that $\eta_t=\rho_t$ for every $t\in[0,1]$ and since clearly $\eta_0=\rho_0$,  this will follow if we show that $(\eta_t)$ solves the continuity equation for $Z_\infty$. Thus let $f\in {\rm Test}(\X)$, recall that $t\mapsto\int f\rho_{n,t}\,\d\mm$ is absolutely continuous with
\begin{equation}
\label{eq:CEn}
\frac{\d}{\d t}\int f\rho_{n,t}\,\d\mm=\int\d f(Z_{n,t})\rho_{n,t}\,\d \mm\qquad a.e.\ t
\end{equation}
and notice that what already proved grants that $\int f\rho_{n,t}\,\d\mm\to \int f\eta_{t}\,\d\mm$ as $n\to\infty$ for every $t\in[0,1]$. Hence to conclude it is sufficient to show that $\int\d f(Z_{n,t})\rho_{n,t}\,\d \mm$ converges to $\int\d f(Z_{\infty,t})\eta_{t}\,\d \mm$ in the sense of distributions. Here and below we shall put $Z_{n,t}\equiv 0$ for $t\notin[0,1]$, $\rho_{n,t}=\rho_0$ for $t\leq 0$ and $\rho_{n,t}=\rho_{n,1}$ for $t\geq 1$, similarly for $\eta_t$.

Now fix $\varphi\in C^\infty_c(\R)$ and let $(\psi_k)\subset C_c(\R)$ be a sequence of functions such that $\int\psi_k=1$ and $\psi_k\geq 0$ and $\supp(\psi_k)\subset [-\frac1k,\frac1k]$ for every $k\in\N$. Then for every $k\in\N$ we have
\[
\begin{split}
\Big|\iint_0^1\varphi_t\Big( \d f(Z_{n,t})\rho_{n,t}-\d f(Z_{\infty,t})\eta_{t}\Big)\,\d t\,\d \mm\Big|\leq&\Big|\iint_0^1\varphi_t\Big( \d f(Z_{n,t})\rho_{n,t}-\d f(Z^{\psi_k}_{n,t})\rho_{n,t}\Big)\,\d t\,\d \mm\Big|\\
&+\Big|\iint_0^1\varphi_t\Big( \d f(Z^{\psi_k}_{n,t})\rho_{n,t}- \d f(Z^{\psi_k}_{\infty,t})\eta_{t}\Big)\,\d t\,\d \mm\Big|\\
&+\Big|\iint_0^1\varphi_t\Big( \d f(Z^{\psi_k}_{\infty,t})\eta_{t}-\d f(Z_{\infty,t})\eta_{t}\Big)\,\d t\,\d \mm\Big|.
\end{split}
\]
Now notice that what previously proved ensures that $(\rho_{n,t})$ converges to $(\eta_t)$ in the weak$^*$ topology of $L^\infty([0,1]\times\X)$, while the assumption and the fact that $|\d f|\in L^\infty(\X)$ grant that $\d f(Z^{\psi_k}_{n,t})\to \d f(Z^{\psi_k}_{\infty,t})$ in $L^1([0,1]\times\X)$ as $n\to\infty$ for every $k\in\N$. Since moreover it is clear that $ \d f(Z^{\psi_k}_{\infty,t})\to  \d f(Z_{\infty,t})$ in $L^1([0,1]\times\X)$ as $k\to\infty$, by letting first $n\to\infty$ and then $k\to\infty$ in the above we obtain:
\[
\begin{split}
\lims_{n\to\infty}\Big|\iint_0^1\varphi_t\Big( \d f(Z_{n,t})\rho_{n,t}-\d f(Z_{\infty,t})\eta_{t}&\Big)\,\d t\,\d \mm\Big|\\
&\leq \lims_{k\to\infty}\sup_{n\in\N}\Big|\iint_0^1\varphi_t\Big( \d f(Z_{n,t})\rho_{n,t}-\d f(Z^{\psi_k}_{n,t})\rho_{n,t}\Big)\,\d t\,\d \mm\Big|.
\end{split}
\]
Thus to conclude the proof it is sufficient to show that the right hand side in the above is $0$. Put for brevity $I_n(t,s):=\int \d f(Z_{n,t})\rho_{n,s}\,\d\mm$ and notice that since $f\in {\rm Test}(\X)$, from \eqref{eq:unifZ} we have that $t\mapsto I_t:=\sup_n|I_n(t,t)|\in L^1(0,1)$ and recalling also \eqref{eq:leibgrad} we see that   $\d f(Z_{n,t})\in W^{1,2}(\X)$ for a.e.\ $t$, with $|\d (\d f(Z_{n,t}))|\leq C(f)(|Z_{n,t}|+|\nabla Z_{n,t}|_{\sf HS})$. Thus from \eqref{eq:CEn} (which is valid for functions $f\in W^{1,2}(\X)$) we obtain
\begin{equation}
\label{eq:lateron}
|I_n(t,s)-I_n(t,t)|\leq\iint_t^s\d(\d f(Z_{n,t}))(Z_{n,r})\rho_{n,r}\,\d r\,\d\mm\leq|s-t|C(f) g_t
\end{equation}
where $g_t:=\sup_n\int_B|Z_{n,t}|^2+|\nabla Z_{n,t}||Z_{n,t}|\,\d\mm$. The assumption  \eqref{eq:unifZ} give  $g\in L^1(0,1)$. Now observe that $\iint_0^1\varphi_t \d f(Z_{n,t})\rho_{n,t}\,\d t\,\d \mm=\iint_0^1\varphi_t\psi_{t-s}I_n(t,t)\,\d s\,\d t$ and that
\[
\begin{split}
\iint_0^1\varphi_t\d f(Z^{\psi_k}_{n,t})\rho_{n,t}\,\d t\,\d \mm=\iiint_0^1\varphi_t\psi^k_{t-s}\d f(Z_{n,s})\rho_{n,t}\,\d s\,\d t\,\d \mm=\iint_0^1\varphi_s\psi^k_{t-s}\,I_n(t,s)\,\d s\,\d t
\end{split}
\]
having used the fact that $\psi^k$ is even in the last step. Therefore using also \eqref{eq:lateron} we get
\[
\begin{split}
\Big|\iint_0^1\varphi_t\Big( \d f(Z_{n,t})\rho_{n,t}-\d f(Z^{\psi_k}_{n,t})\rho_{n,t}\Big)\,\d t\,\d \mm\Big|&=\Big|\iint_0^1\varphi_t\psi^k_{t-s}I_n(t,t)-\varphi_s\psi^k_{t-s}\,I_n(t,s)\,\d s\,\d t  \Big|\\
&\leq \Lip(\varphi)\iint_0^1|s-t|\psi^k_{t-s}|I_n(t,t)|\,\d s\,\d t\\
&\qquad+C(f)\iint_0^1|s-t|\varphi_s\psi^k_{t-s}g_t \,\d s\,\d t.
\end{split}
\]
Recalling that by construction we have $\psi^k_{t-s}=0$ for $|s-t|>\frac2k$ we obtain
\[
\sup_n\Big|\iint_0^1\varphi_t\Big( \d f(Z_{n,t})\rho_{n,t}-\d f(Z^{\psi_k}_{n,t})\rho_{n,t}\Big)\,\d t\,\d \mm\Big|\leq 2\frac{\Lip(\varphi)+C(f)\int_0^1\varphi}k\int_0^1I_t+g_t\,\d t
\]
and the conclusion follows letting $k\to\infty$.
\end{proof}
Our main example of sequence of vector fields converging weakly in time and strongly in space is the one given by the following proposition:
\begin{proposition}\label{prop:sumvect}
Let  $Z^1,Z^2\in \L^1(T\X)$ and for every $n\in\N$ define $Z_n\in L^1([0,1],L^1(T\X))$ by putting
\begin{equation}
\label{eq:zn}
Z_{n,t}:=2Z^{i(j)}\quad\text{\rm for }t\in[\tfrac j{2^n},\tfrac {j+1}{2^n})\ \text{\rm where $i(j)=1$ if $j$ is even and $i(j)=2$ if $j$ is odd}.
\end{equation}
Then $(Z_n)$ converges weakly in time and strongly in space to the vector field constantly equal to $Z^1+Z^2 \in L^1(T\X)$.
\end{proposition}
\begin{proof}
Let $I_{n,i}:=\{t\in[0,1]:Z_{n,t}=2Z^i\}$, $i=1,2$, so that we can write $Z_n=2\nchi_{I_{n,1}}Z^1+2\nchi_{I_{n,2}}Z^2$, being intended that $Z_{n,t}=0$ for $t\notin[0,1]$. Then for $\psi\in C^\infty_c(\R)$ directly from the definition we see that $Z^\psi_{n}=2\nchi_{I_{n,1}}\ast\psi Z^1+2\nchi_{I_{n,2}}\ast\psi Z^2$. Now notice that Young's inequality $\|f\ast\psi\|_{L^1}\leq\|f\|_{L^1}\|\psi\|_{L^1}$ ensures that the operator `convolution with $\psi$' is continuous from $L^1(\R)$ into itself, thus since  $(\nchi_{I_{n,1}})$ converges to $\frac12\nchi_{[0,1]}$ weakly in $L^1(\R)$, we see that  $(\nchi_{I_{n,1}}\ast\psi)$ weakly converges to $\frac12\nchi_{[0,1]}\ast\psi$ in $L^1(\R)$. Moreover, since the functions  $\nchi_{I_{n,1}}\ast\psi$ are uniformly Lipschitz and with uniformly bounded support, they  form a  relatively compact subset in $L^1(\R)$, hence the convergence of  $(\nchi_{I_{n,1}}\ast\psi)$ to $\frac12\nchi_{[0,1]}\ast\psi$ is strong in $L^1(\R)$. From this it easily follows that the sequence $(2\nchi_{I_{n,1}}\ast\psi Z^1)$ converges to $ \nchi_{[0,1]}\ast\psi Z^1$ strongly in $L^1([0,1],L^1(T\X))$ and since a similar argument works for $(2\nchi_{I_{n,2}}\ast\psi Z^2)$, the claim follows.
\end{proof}
For later use, we notice the following:
\begin{proposition}\label{prop:pertr}
Let $Z^1,Z^2$ be two regular vector fields, define $Z_n$ as in \eqref{eq:zn} and let $\rho_0$ be a bounded probability density with bounded support. Let $\rho_t$ be the density of $(\Fl^{Z^1+Z^2}_t)_*(\rho_0\mm)$  and define $\rho^1_{n,t},\rho^2_{n,t}$ by
\[
\begin{split}
\rho^1_{n,t}\mm&:=(\Fl^{Z_1}_{t-\frac {2i}{2^n}})_*(\Fl^{Z_n}_{\frac {2i}{2^n}})_*(\rho_0\mm)\qquad\text{\rm for }t\in[\tfrac{2i}{2^n},\tfrac{2(i+1)}{2^n}),\\
\rho^2_{n,t}\mm&:=(\Fl^{Z_1}_{t-\frac {2i+1}{2^n}})_*(\Fl^{Z_n}_{\frac {2i+1}{2^n}})_*(\rho_0\mm)\qquad\text{\rm for }t\in[\tfrac{2i+1}{2^n},\tfrac{2(i+1)+1}{2^n}).
\end{split}
\]
Then for every $t\in[0,1]$ both the sequences $(\rho^1_{n,t}),(\rho^2_{n,t})$ converge to $\rho_t$ in the weak$^*$ topology of $L^\infty(\X)$.
\end{proposition}
\begin{proof}
We shall prove the result for $\rho^1_{n,t}$ only, as the study of $\rho^2_{n,t}$ follows along similar lines.  Fix $t\in[0,1]$ and let $i_n\in\N$ be such that $t\in[\tfrac{2i_n}{2^n},\tfrac{2(i_n+1)}{2^n})$.

The estimates \eqref{eq:bcompr} and \eqref{eq:ms} grant that the densities $\{\rho^1_{n,t}\}_n$ are equibounded, hence to conclude it is sufficient to prove that $\int \varphi\rho^1_{n,t}\,\d\mm\to\int\varphi\rho_t\,\d\mm$ for a set $\varphi$'s dense in $L^1(\X)$. We shall consider $\varphi$ bounded and Lipschitz and notice that
\[
\begin{split}
\Big|\int \varphi\rho^1_{n,t}\,\d\mm-\int\varphi\rho_t\,\d\mm\Big|\leq&\underbrace{\Big|\int \varphi\,\d (\Fl^{Z_1}_{t-\frac {2i_n}{2^n}})_*(\Fl^{Z_n}_{\frac {2i_n}{2^n}})_*(\rho_0\mm)-\int\varphi\,\d(\Fl^{Z_n}_{t})_*(\rho_0\mm)\Big|}_{A_n}\\
&\qquad\qquad\qquad\qquad\qquad+\underbrace{\Big|\int\varphi\,\d(\Fl^{Z_n}_{t})_*(\rho_0\mm)-\int\varphi\rho_t\,\d\mm\Big|}_{B_n}.
\end{split}
\]
Theorem \ref{thm:stabfl} ensures that $B_n\to0$ as $n\to\infty$, for $A_n$ we put $\mu_n:=(\Fl^{Z_n}_{\frac {2i_n}{2^n}})_*(\rho_0\mm)\in\mathscr P(\X)$,  $\sigma_n:=(\Fl^{Z^1}_{t-\frac {2i_n}{2^n}},\Fl^{Z_n}_{t-\frac {2i_n}{2^n}} )_*\mu_n\in\mathscr P(\X^2) $ and notice that
\[
\begin{split}
A_n&=\Big|\int\varphi\,\d(\pi_1)_*\sigma_n-\int\varphi\,\d(\pi_2)_*\sigma_n\Big|=\Big|\int\varphi(x)-\varphi(y)\,\d\sigma_n(x,y)\Big|\leq\Lip(\varphi)\int \sfd(x,y)\,\d\sigma_n(x,y).
\end{split}
\]
Thus recalling, from \eqref{eq:ms}, that for $\mm$-a.e.\ $x$ we have
\begin{equation}
\label{eq:estdist}
\begin{split}
\sfd(\Fl^{Z^1}_{t-\frac {2i_n}{2^n}}(x),x)&\leq 2\|Z^1\|_{L^\infty} \big|t-\frac {2i_n}{2^n}\big|\leq 2^{2-n}\|Z^1\|_{L^\infty} ,\\
\sfd(\Fl^{Z_n}_{t-\frac {2i_n}{2^n}}(x),x)&\leq 2\max\{\|Z^1\|_{L^\infty},\|Z_n\|_{L^\infty}\} \big|t-\frac {2i_n}{2^n}\big|\leq 2^{2-n}\max\{\|Z^1\|_{L^\infty},\|Z_n\|_{L^\infty}\},
\end{split}
\end{equation}
the conclusion follows from
\[
\begin{split}
\int \sfd(x,y)\,\d\sigma_n(x,y)&=\int \sfd\big(\Fl^{Z_1}_{t-\frac {2i}{2^n}}(x),\Fl^{Z_n}_{t-\frac {2i}{2^n}}(x)\big)\,\d\mu_n (x)\\
&\leq \int \sfd\big(\Fl^{Z_1}_{t-\frac {2i_n}{2^n}}(x),x\big)\,\d\mu_n (x)+\int \sfd\big(x,\Fl^{Z_n}_{t-\frac {2i_n}{2^n}}(x)\big)\,\d\mu_n(x)\\
\text{\rm by \eqref{eq:estdist}}\qquad\qquad&\leq2^{3-n}\max\{\|Z^1\|_{L^\infty},\|Z^2\|_{L^\infty}\}.
\end{split}
\]
\end{proof}

\subsection{The Euclidean case}
The arguments used in the previous section can also be used in the Euclidean context to extend known stability results for vector fields converging weakly in time and strongly in space, to the $BV$ case. Compare with \cite[Remark 5.11]{Ambrosio2008}.

The following is a rather trivial observation:
\begin{lemma}\label{le:bvrn}
Let $(\rho_t,v_t)$ be a solution of the continuity equation on $\R^d$ with $(\rho_t)\in L^\infty([0,1],L^\infty(\R^d))$, $(|v_t|)\in L^1([0,1],L^\infty(\R^d))$ and $f\in BV(\R^d)$.

Then $t\mapsto\int f\rho_t\,\d\mathcal L^d$ is absolutely continuous and for its derivative it holds
\[
\Big|\frac{\d}{\d t}\int f\rho_t\,\d\mathcal L^d\Big|\leq \|Df \|_{\sf TV}\|\rho_t\|_{L^\infty}\|v_t\|_{L^\infty}\qquad a.e.\ t.
\]
\end{lemma}
\begin{proof}
For $f\in C^{\infty}_c(\R^d)$ the claim is a direct consequence of the distributional formulation of the continuity equation, which ensures that
\[
\frac{\d}{\d t}\int f\rho_t\,\d\mathcal L^d=\int \d f(v_t)\rho_t\,\d\mathcal L^d\qquad a.e.\ t.
\]
Then the conclusion follows recalling that $\|Df \|_{\sf TV}=\| |\d f|\|_{L^1 }$. The general case follows from a standard approximation procedure; we omit the details.
\end{proof}
With this last lemma and adapting the arguments used for Theorem \ref{thm:stabfl} we deduce the following result:
\begin{theorem}\label{thm:stabfleu} Assume that $(Z_n)\subset L^1([0,1],L^1(\R^d,\R^d;\mathcal L^d))$ converges weakly in time and strongly in space to $Z_\infty$. Assume also that
\begin{equation}
\label{eq:unifZ2}
\int_0^1\sup_{n\in \N\cup\{\infty\}}\||Z_{n,t}|\|_{L^\infty}+\| D Z_{n,t}\|_{\sf TV} \,\d t+\sup_{n\in \N\cup\{\infty\}}\int_0^1\|(\div Z_{n,t})^-\|_{L^\infty}\,\d t<\infty.
\end{equation}
Then the Regular Lagrangian Flows $(\Fl^{Z_n})$ converge locally in measure to the Regular Lagrangian Flow $\Fl^{Z_\infty}$ and, for every $t\in[0,1]$, the maps $(\Fl^{Z_n}_t)$ converge locally in measure to $\Fl^{Z_\infty}_t$.
\end{theorem}
\begin{proof} The argument is verbatim the same used in the proof of Theorem \ref{thm:stabfl}, with  the following differences:   the function $f$ is taken in $C^\infty_c(\R^d)$, so that  the assumption \eqref{eq:unifZ2} yields that  $\d f(Z_{n,t}) \in BV(\R^d)$ and taking into account Lemma \ref{le:bvrn} above we obtain the estimate
\[
\begin{split}
|I_n(t,s)-I_n(t,t)|&=\Big|\int_t^s\frac{\d}{\d r}\int \d f(Z_{n,t})\rho_{n,r}\,\d\mathcal L^d\,\d r \Big|\leq \|D( \d f(Z_{n,t}))\|_{\sf TV}\int_t^s\|\rho_{n,r}\|_{L^\infty}\|Z_{n,r}\|_{L^\infty}\,\d r.\\
\end{split}
\]
Using this estimate in place of \eqref{eq:lateron}, the conclusion is obtained arguing as in Theorem \ref{thm:stabfl}.
\end{proof}

\section{Directional Energy}

\subsection{Basic considerations about the space $L^p(\X,\Y_{\bar y})$}
In this short section we collect some basic simple properties of the space $L^p(\X,\Y_{\bar y})$. Let us fix a complete and separable metric space $(\X,\sfd)$ equipped with a non-negative Radon measure $\mm$ giving finite mass to bounded sets and a pointed complete space $(\Y,\sfd_\Y,\bar y)$ (which often, but not always, will be separable).
\bigskip

The behaviour of $AC^p$ curves with values in $L^p(\X,\Y_{\bar y})$ is described in the following lemma (compare with  Lemma \ref{le:basequ}):
\begin{lemma}\label{le:basequ2}
Let $(x,t)\mapsto f_t(x)\in \Y$ be a given Borel map in $L^p(\X\times[0,1],\Y_{\bar y})$ and $p\in(1,\infty)$. Then the following are equivalent:
\begin{itemize}
\item[i)] The curve $[0,1]\ni t\mapsto f_t\in L^p(\X,\Y)$ belongs to $W^{1,p}([0,1],L^p(\X,\Y_{\bar y}))$ (resp.\  $AC^{p}([0,1],L^p(\X,\Y_{\bar y}))$).
\item[ii)] There is a function $G\in L^p(\X\times[0,1])$, $G\geq 0$, such that for a.e.\ (resp.\ every) $t,s\in[0,1]$, $t<s$ it holds
\begin{equation}
\label{eq:geq2}
\sfd_\Y(f_s,f_t)\leq \int_t^s G_r\,\d r\quad\mm-a.e..
\end{equation}
\item[iii)] For $\mm$-a.e.\ $x\in\X$ we have $f_\cdot(x)\in W^{1,p}([0,1],\Y)$ and the function $(t,x)\mapsto |\partial_tf_t(x)|=:H_t(x)$ belongs to $L^p(\X\times[0,1])$ (resp.\ and moreover $(f_t)\in C([0,1],L^p(\X,\Y_{\bar y}))$).
\end{itemize}
Moreover, if these holds $H$ is the least function $G\geq 0$ in the $\mm\times\mathcal L^1$-a.e.\ sense, for which \eqref{eq:geq2} holds and
\begin{equation}
\label{eq:splp}
\frac{\sfd_\Y(f_{t+h}(x),f_t(x))}{|h|}\quad\to\quad H_t(x)\qquad\text{ in $L^p(\X\times[0,1])$ as $h\to 0$,}
\end{equation}
where the incremental ratios are defined to be 0 if $t+h\notin[0,1]$.
\end{lemma}
\begin{proof}\ We shall deal with the absolutely continuous case, as the Sobolev one can be obtained through very similar arguments taking also into account Theorem \ref{thm:sobaccurve}. Moreover, since $(f_t)\in L^p(\X\times[0,1],\Y_{\bar y})$, by definition there is a separable subset of $\Y$ containing, up to negligible sets, the image of $(f_t)$; thus up to replacing $\Y$ with the closure of such separable subset we can assume that $\Y$ is separable. Hence, without loss of generality we may assume that $L^{p}(\X,\Y_{\bar y})$ is separable.

\noindent{$\mathbf{(i)\Rightarrow(ii)}$} Let $(f_n)\subset L^p(\X,\Y_{\bar y})$ be countable and dense, put $F_{n,t}:=\sfd_\Y(f_t,f_n)\in L^p(\X)$ and notice that the triangle inequality in $\Y$ gives $|F_{n,s}-F_{n,t}|\leq \sfd_\Y(f_t,f_s)$ $\mm$-a.e.. On the other hand, the triangle inequality in $L^p(\X)$ gives $\|F_{n,s}-F_{n,t}\|_{L^p(\X)}\geq \sfd_{L^p(\X,\Y)}(f_s,f_n)-\sfd_{L^p(\X,\Y)}(f_t,f_n)$, so that the identity $\sfd_{L^p(\X,\Y)}(f_s,f_t)=\sup_n\sfd_{L^p(\X,\Y)}(f_s,f_n)-\sfd_{L^p(\X,\Y)}(f_t,f_n)$ (consequence of the density of the $f_n$'s) forces
\begin{equation}
\label{eq:densefn}
\sfd_Y(f_s,f_t)=\sup_n |F_{n,s}-F_{n,t}|\qquad\mm-a.e.\ \forall t,s\in[0,1].
\end{equation}
In particular, for every $n\in\N$ we have $(F_{n,t})\in AC^p([0,1],L^p(\X))$ so that by the Radon-Nikodym property of $L^p$ we obtain that $G_{n,t}:=\partial_tF_{n,t}$ is a well defined function in $L^p(\X)$ for a.e.\ $t\in[0,1]$.

Now observe that the assumption  $(f_t)\in AC^p([0,1],L^p(\X,\Y_{\bar y}))$ ensures (by arguing as in \eqref{eq:oneside}) that given a sequence $h_{i} \downarrow 0$ the incremental ratios $\frac{\sfd_\Y(f_{t+h_{i}},f_t)(x)}{h_{i}}$ are bounded in $L^p(\X\times[0,1])$ as $h_{i} \to 0$, hence up to subsequences they must converge to a limit $\tilde G$ weakly in $L^p(\X\times[0,1])$. Thus \eqref{eq:densefn} forces $|G_{n,t}|(x)\leq\tilde G_t(x)$ for $\mm\times\mathcal L^1$-a.e.\ $(x,t)$ and in turn this grants that $G:=\sup_n|G_n|$ belongs to $L^p(\X\times[0,1])$. The conclusion follows noticing that for any $t,s\in[0,1]$, $t<s$ it holds
\[
\sfd_Y(f_s,f_t)=\sup_n |F_{n,s}-F_{n,t}|\leq \sup_n\int_t^s|G_{n,r}|\,\d r\leq\int_t^s G_r\,\d r,\qquad\mm-a.e.,
\]
as desired.

\noindent{$\mathbf{(ii)\Rightarrow(i)}$} Directly from \eqref{eq:geq2} we obtain
\[
\sfd_{L^p(\X,\Y)}(f_s,f_t)=\|\sfd_\Y(f_s,f_t)\|_{L^p(\X)}\leq\Big\|\int_t^s G_r\,\d r\Big\|_{L^p(\X)}\leq\int_t^s \|G_r\|_{L^p(\X)}\,\d r
\]
for any $t,s\in[0,1]$, $t<s$, and since the identity $\int_0^1\|G_t\|_{L^p(\X)}^p\,\d t=\iint_0^1|G_t|^p\,\d t\,\d\mm<\infty$ shows that $(\|G_t\|_{L^p})\in L^p(0,1)$, this is sufficient to conclude.

\noindent{$\mathbf{(ii)\Rightarrow(iii)}$} Continuity is obvious from the implication ${(ii)\Rightarrow(i)}$ already proved. For any 1-Lipschitz function $\varphi:\Y\to \R$ the function $\varphi\circ f$ satisfies $(i)$ of Lemma \ref{le:basequ} with $g:=G$ and thus Lemma \eqref{le:basequ} ensures that for $\mm$-a.e.\ $x\in\X$ the function $t\mapsto\varphi\circ f_t(x)$ belongs to $W^{1,p}([0,1])$ and its distributional derivative is bounded above by $G_t(x)$. Letting $\varphi$ running over the countable set given by Lemma \ref{le:spcount} we conclude that $t\mapsto f_t(x)$ belongs to $W^{1,p}([0,1],\Y)$ with distributional derivative bounded above by $G_t(x)$ for $\mm$-a.e.\ $x$.

\noindent{$\mathbf{(iii)\Rightarrow(ii)}$} It is trivial to notice that for any 1-Lipschitz function $\varphi:\Y\to \R$ the function $\varphi\circ f$ belongs to $W^{1,p}((0,1))$ and $\partial_t(\varphi(f_t(x)))=: h^{\varphi}_{t}(x)\leq H_{t}(x)$. Thus Lemma \eqref{le:basequ} ensures that for every  $t,s$ and $\mm$-a.e.\ $x\in\X$ it holds
\[
|\varphi(f_s(x))-\varphi(f_t(x))|\leq\int_t^s H_r(x)\,dr
\]
hence taking the supremum as $\varphi$ varies over the countable family given by Lemma \ref{le:spcount}  we conclude.

\noindent{\bf Final statements.} The fact that $H$ is the minimal $G$ for which \eqref{eq:geq2} holds follows directly from the proof given. For what concerns \eqref{eq:splp}, notice that \eqref{eq:geq2} and the choice $G:=H$  give
\[
\frac{\sfd_\Y(f_{t+h}(x),f_t(x))}{|h|}\leq \frac1{|h|}\int_t^{t+h}H_r(x)\,\d r\quad\to\quad H_t(x)\qquad\text{ in $L^p(\X\times[0,1])$ as $h\to 0$,}
\]
where the claimed convergence is an easy consequence of the definition of Bochner integral. Hence   $\lims_{h\to 0}\|\frac{\sfd_\Y(f_{\cdot+h}(\cdot),f_\cdot(\cdot))}{|h|}\|_{L^p}\leq \|H\|_{L^p}$. Now let $\tilde H\in L^p$ be any $L^p$-weak limit of $\frac{\sfd_\Y(f_{\cdot+h}(\cdot),f_\cdot(\cdot))}{|h|}$ along some sequence  $h_n\to 0$, so that $\|\tilde H\|_{L^p}\leq\|H\|_{L^p}$, and notice that to conclude it is sufficient to prove that $\tilde H=H$. For any  $\varphi:\Y\to \R$ 1-Lipschitz, the function $\varphi\circ f$ satisfies \eqref{eq:geq} with $g:=H$, hence putting $h_\varphi(x):=\partial_t(\varphi( f_t(x)))$ as before, from the trivial bound
\[
\frac{\sfd_\Y(f_{t+h}(x),f_t(x))}{|h|}\geq \frac{|\varphi(f_{t+h}(x))-\varphi(f_t(x))|}{|h|}
\]
we deduce $\tilde H\geq h_\varphi$. Hence letting $\varphi$ run in the countable set given by Lemma  \ref{le:spcount} we deduce that $\tilde H\geq H$ and then the conclusion.
\end{proof}
We now want to prove a continuity result for $L^p$ functions valued in $\Y$ and to this aim it is convenient to first analyze the case $\Y:=\ell^\infty$.
\begin{lemma}\label{le:cbdense}
For every $p\in[1,\infty)$, the space $C_b(\X,\ell_\infty)$ is dense in $L^p(\X,\ell_\infty)$.
\end{lemma}
\begin{proof}
Let $E\subset\X$ Borel, $f\in\ell_\infty$ and $(g_n)\subset C_b(\X,\R)$ be converging to $\nchi_E$ in the $L^p(\X,\R)$-topology. Then $(g_nf)\subset C_b(\X,\ell_\infty)$ converges to $\nchi_Ef$ in the topology of $L^p(\X,\ell_\infty)$. Since linear combinations of functions of the form $\nchi_Ef$ with $E,f$ as above are dense in $L^p(\X,\ell_\infty)$  (recall \eqref{eq:simpledense}), the proof is completed.
\end{proof}

\begin{proposition}\label{le:contlpfl}
 Let $p\in[1,\infty)$, $Z$ a time dependent regular vector field on $\X$ and $u\in L^p(\X,\Y_{\bar y})$. Then the map $\R\ni t\mapsto u\circ\Fl^Z_t\in L^p(\X,\Y_{\bar y})$ is continuous.
\end{proposition}
\begin{proof}
Up to a left composition with a (Kuratowski) isometric embedding of $\Y$ in $\ell_\infty$ we can assume that $\Y=\ell_\infty$. Then observe that the trivial bound
\[
\int\sfd_\Y^p(u\circ\Fl^Z_t,v\circ\Fl^Z_t)\,\d\mm\leq\int \sfd_\Y^p(u,v)\,\d (\Fl^Z_t)_*\mm\leq e^{\int_0^t\|(\div Z_t)^-\|_{L^\infty}\,\d t}\int \sfd_\Y^p(u,v)\,\d \mm
\]
shows that the right composition with $\Fl^Z_t$ is a Lipschitz map from $L^p(\X,\ell_\infty)$ to $L^p(\X,\ell_\infty)$, thus to conclude it is sufficient to prove that there is a dense subset of $L^p(\X,\ell_\infty)$ made of functions $u$ such that $t\mapsto u\circ\Fl^Z_t\in L^p(\X,\ell_\infty)$ is continuous. An application of the dominated convergence theorem shows that this is the case for $u\in C_b(\X,\ell_\infty)$, thus the conclusion follows from Lemma \ref{le:cbdense}.
\end{proof}

\subsection{The Korevaar-Schoen space $\KS Z(\Omega,\Y_{\bar y})$}\label{se:ks}

Let us fix some regular vector field $Z$ not depending on time on the $\RCD(K,\infty)$ space $(\X,\sfd,\mm)$ and denote by $\Fl^{Z}$ the unique regular Lagrangian flow  associated to $Z$. Also, let $(\Y,\sfd_\Y,\bar y)$ be a pointed complete space.

Let  $p \in (1,\infty)$,  $u \in L^{p}(\X,\operatorname{Y}_{\bar y})$, $\Omega\subset\X$ open and   $\varepsilon > 0$ we set
\[
e^{Z}_{p,\varepsilon}[u,\Omega](x): =
\begin{cases}
\dfrac{\operatorname{d}^{p}_{\operatorname{Y}}\bigl(u(x), u(\Fl^{Z}_\eps(x))\bigr)}{\varepsilon^{p}},& \quad \text{\rm if }\ x, \Fl^{Z}_\eps(x) \in \Omega;\\
0, &\quad \hbox{otherwise}.
\end{cases}
\]
For every $\varphi \in C_{b}(\X)$ we set
\[
E^{Z}_{p,\varepsilon}[u,\Omega](\varphi):=\int \varphi(x) e^{Z}_{p,\varepsilon}[u,\Omega](x)\,\d\mathfrak{m}(x).
\]

\begin{definition}
\label{Def3.10}
We say that a Borel map $u:\X\to\Y$ belongs to Korevaar-Schoen space $\KS Z (\Omega,\operatorname{Y}_{\bar y})$ if  $u \in L^{p}(\X,\operatorname{Y}_{\bar y})$ and
\begin{equation}
\label{eq3.20}
E^{Z}_{p}[u,\Omega]:=\sup\varlimsup\limits_{\varepsilon \to 0}E^{Z}_{p,\varepsilon}[u,\Omega](\varphi)<+\infty.
\end{equation}
and the $\sup$ is taken among all $\varphi\in C_b(\X)$ with $0\leq \varphi\leq1$ and $\sfd(\supp(\varphi),\Omega^c)>0$ (if $\Omega=\X$ we interpret this last condition as automatically satisfied).

The quantity $E^{Z}_{p}[u,\Omega]$ will be called  energy of a map $u$ in the direction $Z$ on $\Omega$.
\end{definition}

\begin{theorem}\label{thm:baseks} Let $(\X,\sfd,\mm)$ be a $\RCD(K,\infty)$ space, $(\Y,\sfd_\Y,{\bar y})$ a pointed complete space, $p\in(1,\infty)$, $Z$ a regular vector field on $\X$, $\Omega\subset \X$ open and $u\in L^p(\Omega,\Y_{\bar y})$.

For $x\in\Omega$ we put $T_x:=\frac{1}{\|Z\|_{L^\infty}\sfd(x,\Omega^c)}$ (if $\Omega=\X$ we put $T_x:=+\infty$) and for $C\subset\Omega$ closed put
\begin{equation}
\label{eq:deftc}
T_C:=\inf_{x\in C}T_x.
\end{equation}
Then the following are equivalent:
\begin{itemize}
\item[i)] It holds $u\in \KS Z(\Omega,\Y_{\bar y})$.
\item[ii)] For every closed set $C\subset \Omega$ with  $T_C>0$  the curve $[0,1\wedge T_C)\ni t\mapsto u\circ\Fl^Z_t\in L^p(C,\Y_{\bar y})$ is Lipschitz with Lipschitz constant independent on $C$.
\item[iii)] There exists $G\in L^p(\Omega)$ such that the following holds. For every  closed set $C\subset \Omega$ with $T_C>0$  we have
\begin{equation}
\label{eq:g2}
\sfd_\Y(u\circ\Fl^Z_s,u\circ\Fl^Z_t)\leq \int_t^s G\circ\Fl^Z_r\,\d r\quad\mm-a.e.\ \text{\rm on }C\qquad\forall t,s\in[0,T_C),\ t\leq s,
\end{equation}
(and in particular the map $t\mapsto u\circ\Fl^Z_t$ belongs to $AC^p_{\rm loc}([0,T_C),L^p(C,\Y_{\bar y}))$).
\item[iv)]  For $\mm$-a.e.\ $x\in \Omega$ the map $t\mapsto u(\Fl^Z_t(x))$ belongs to $W^{1,p}([0,T_x],\Y)$  and for some  $H\in L^p(\Omega)$ the  distributional derivative $|\partial_tu(\Fl^Z_t(x))|$ satisfies the identity
\begin{equation}
\label{eq:g3}
|\partial_tu(\Fl^Z_t(x))|= H(\Fl^Z_t(x))\qquad a.e.\ t\in[0,T_x].
\end{equation}
\end{itemize}
Moreover if these hold the functions $(e^{Z}_{p,\varepsilon}[u,\Omega])^{1/p}$ converge to nonnegative $H$ in $L^p(\Omega)$ as $\eps\downarrow0$,   we have
\begin{equation}
\label{eq:energy}
E^{Z}_{p}[u,\Omega]=\int_\Omega |H|^p\,\d\mm,
\end{equation}
and the choice $G:=H$ is admissible in \eqref{eq:g2} and provides the least, in the $\mm$-a.e.\ sense, function $G\geq 0$ for which \eqref{eq:g2} holds.
\end{theorem}
\begin{proof}\ \\

\noindent{$\mathbf{(iv) \Rightarrow(iii)}$ } By Proposition \ref{le:contlpfl} we know that $t\mapsto u\circ \Fl^Z_t\in L^p(C,\Y_{\bar y})$ is continuous. Then the conclusion follows from Lemma \ref{le:basequ2}.

\noindent{$\mathbf{(iii) \Rightarrow(ii)}$ } The bound
\[
\begin{split}
\int_C\frac{\sfd_\Y^p(u\circ \Fl^Z_s,u\circ \Fl^Z_t)}{|s-t|^p}\,\d\mm\stackrel{\eqref{eq:g2}}\leq&\frac1{|s-t|}\int_C\int_t^s G^p\circ\Fl^Z_r\,\d r\,\d\mm\\
\stackrel{\phantom{\eqref{eq:g2}}}=&\frac1{|s-t|}\int_t^s\int G^p \,\d(\Fl^Z_r)_*(\mm\restr C)\,\d r\stackrel{\eqref{eq:bcompr}}\leq e^{|s-t|\|(\div Z)^-\|_{L^\infty}}\int_\Omega G^p\,\d\mm
\end{split}
\]
yields that the Lipschitz constant of $[0,1\wedge T_C)\ni t\mapsto u\circ\Fl^Z_t\in L^p(C,\Y_{\bar y})$ is bounded from above by $e^{\frac1p\|(\div Z)^-\|_{L^\infty}}\|G\|_{L^p(\Omega)}$ and in particular is independent on $C$, as desired.

\noindent{$\mathbf{(ii) \Rightarrow(i)}$ } Let $L$ be the uniform Lipschitz constant of $[0,1\wedge T_C)\ni t\mapsto u\circ\Fl^Z_t\in L^p(C,\Y_{\bar y})$. Now pick  $\varphi\in C_b(\X)$ with $0\leq \varphi\leq1$ and $\sfd(\supp(\varphi),\Omega^c)>0$, put $C:=\supp(\varphi)$ and simply notice that
\[
\int  \varphi \,e^{Z}_{p,\varepsilon}[u,\Omega]\,\d\mm\leq\int_C\frac{\sfd^p_\Y(u\circ\Fl^Z_\eps,u)}{\eps^p}\,\d\mm\leq  L^p ,
\]
so that  the claim follows letting $\eps\downarrow0$.

\noindent{\bf $\mathbf{(i) \Rightarrow(iv)}$} Let $\alpha>\beta>0$ be two parameters and consider the closed set $C_\alpha\subset \Omega$ defined as $C_\alpha:=\{x\in\Omega:\sfd(x,\Omega^c)\}\geq \alpha\|Z\|_{L^\infty}\}$ (if $\Omega=\X$ we pick $C_\alpha=\X$ as well - if $Z=0$ the set $C_\alpha$ might be not closed but in this case the claim is trivial). We start claiming that
\begin{equation}
\label{eq:claimi-iii}
\begin{split}
\frac{e^{-\alpha\|(\div Z)^-\|_{L^\infty}}}{\alpha-\beta}\int_{C_\alpha}\int_0^{\alpha-\beta}\sfd_\Y^p(u\circ\Fl^Z_t,\bar y)\,\d t\,\d\mm&\leq \int_\Omega\sfd_\Y^p(u,\bar y)\,\d\mm\qquad\forall \bar y\in\Y,\\
\frac{e^{-\alpha\|(\div Z)^-\|_{L^\infty}}}{\alpha-\beta}\int_{C_\alpha}\mathcal E_{p,[0,\alpha-\beta]}(u\circ\Fl^Z_\cdot) \,\d\mm&\leq E^{Z}_{p}[u,\Omega].
\end{split}
\end{equation}
To check the first, notice that
\[
\begin{split}
\int_{C_\alpha}\int_0^{\alpha-\beta}\sfd_\Y^p(u\circ\Fl^Z_t,\bar y)\,\d t\,\d\mm&=\int_0^{\alpha-\beta}\int\sfd_\Y^p(u,\bar y)\,\d(\Fl^Z_t)_*(\mm|_{C_\alpha})\,\d t\\
&\leq({\alpha-\beta}) e^{\alpha\|(\div Z)^-\|_{L^\infty}}\int_{\Omega}\sfd_\Y^p(u,\bar y)\,\d \mm.
\end{split}
\]
For the second,  start observing that Corollary \ref{cor:energy0} and the monotone convergence theorem gives
\[
\begin{split}
\int_{C_\alpha}\mathcal E_{p,[0,\alpha-\beta]}(u\circ\Fl^Z_\cdot) \,\d\mm&=\lim_{\eps\downarrow 0}\int_{C_\alpha}\int_0^{\alpha-\beta-\eps}\frac{\sfd_\Y^p(u\circ\Fl^Z_{t+\eps},u\circ\Fl_t^Z)}{\eps^p}\,\d t\,\d\mm\\
&=\lim_{\eps\downarrow 0}\int_0^{\alpha-\beta-\eps}\int\frac{\sfd_\Y^p(u\circ\Fl^Z_{\eps},u)}{\eps^p}\,\d(\Fl_t^Z)_*(\mm|_{C_\alpha})\,\d t\\
&\leq({\alpha-\beta}) e^{\alpha\|(\div Z)^-\|_{L^\infty}}\limi_{\eps\downarrow 0}\int_{C_{\alpha-\beta}}\frac{\sfd_\Y^p(u\circ\Fl^Z_{\eps},u)}{\eps^p}\,\d \mm.
\end{split}
\]
Now  let $\varphi\in C_b(\X)$ be with $\supp(\varphi)\subset\Omega$ and $\varphi\equiv 1$ on $C_{\alpha-\beta}$ and notice that
\[
\limi_{\eps\downarrow 0}\int_{C_{\alpha-\beta}}\frac{\sfd_\Y^p(u\circ\Fl^Z_{\eps},u)}{\eps^p}\,\d \mm\leq\lims_{\eps\downarrow 0}\int_\Omega\varphi\frac{\sfd_\Y^p(u\circ\Fl^Z_{\eps},u)}{\eps^p}\,\d \mm\leq E^{Z}_{p}[u,\Omega],
\]
thus our claim \eqref{eq:claimi-iii} is proved. It follows that
\begin{equation}
\label{eq:acae}
\textrm{ for $\mm$-a.e.\ $x\in C_\alpha$ the curve $[0,\alpha-\beta]\ni t\mapsto u(\Fl^Z_t(x))\in \Y$
 belongs to $W^{1,p}([0,\alpha-\beta],\Y)$. }
\end{equation}
Denote its distributional derivative by $t\mapsto F_{\alpha,\beta,t}(x)$ and notice that by the last part of Theorem \ref{thm:sobaccurve} the $\mm\times\mathcal L^1$-a.e.\ defined function $F_{\alpha,\beta}:[0,\alpha-\beta]\times C_\alpha\to\R$ is Borel and by the second in \eqref{eq:claimi-iii} belongs to $L^p([0,\alpha-\beta]\times C_\alpha)$.

 From the trivial identity $\sfd_\Y(u\circ\Fl^Z_{s+h},u\circ\Fl^Z_{s})=\sfd_\Y(u\circ\Fl^Z_{t+h},u\circ\Fl^Z_{t})\circ\Fl_{s-t}^Z$ it follows that $F_{\alpha,\beta,s}=F_{\alpha,\beta,t}\circ\Fl^Z_{s-t}$ $\mm$-a.e.\ for a.e.\ $s,t$, $s\geq t$. Thus letting $\mu_{\alpha,\beta}:=\Fl^Z_*(\mathcal L^1\restr{[0,\alpha-\beta]}\times\mm\restr C_\alpha)$ we see that there is a $\mu_{\alpha,\beta}$-a.e.\ uniquely defined Borel function $\bar F_{\alpha,\beta}:\Omega\to\R$ such that for a.e.\ $t\in[0,\alpha-\beta]$ it holds
\[
F_{\alpha,\beta,t}(x)=\bar F_{\alpha,\beta}(\Fl^Z_{t}(x))\quad \mathcal L^1\restr{[0,\alpha-\beta]}\times\mm\restr {C_\alpha}-a.e.\ (t,x).
\]
Notice  that since $\mathcal L^1\restr{[0,\alpha-\beta]\cap [0,\alpha'-\beta']}\times\mm\restr {C_\alpha \cap C_{\alpha'}}$-a.e.\ it holds $F_{\alpha,\beta,\cdot}(\cdot)=F_{\alpha',\beta',\cdot}(\cdot)$ , we have that
\begin{equation}
\label{eq:minmu}
\text{$\Fl^Z_*(\mathcal L^1\restr{[0,\alpha-\beta] \cap [0,\alpha'-\beta']}\times\mm\restr {C_\alpha \cap C_{\alpha'}})$-a.e.\ it holds $\bar F_{\alpha,\beta}=\bar F_{\alpha',\beta'}$}.
\end{equation}
Now observe that
\begin{equation}
\label{eq:mumm}
\text{For $E\subset\Omega$ Borel we have $\mm(E)=0$ if and only if $\mu_{\alpha,\beta}(E)=0$,}
\end{equation}
for every $\alpha>\beta>0$.

The 'easy' implication that $\mm(E)=0$ implies $\mu_{\alpha,\beta}(E)=0$ for every $\alpha,\beta$ obviously follows from definitions. To prove the converse implication we proceed as follows. Let $\rho_t$ be the density of $(F_t)_\ast \mm$ w.r.t. $\mm$, so that the functions $\rho_t$ are uniformly bounded in $L^\infty$ for $t\in[0,1]$. The measures $(F_t)_\ast \mm$ converge to $\mm$ weakly in duality with continuous functions with bounded support on $\X$  as $t\downarrow0$ (by the dominated convergence theorem and because the flow is concentrated on continuous curves). This weak convergence plus the uniform $L^\infty$ bound imply that $\rho_t$ converge to 1 in the weak$^\ast$ topology of $L^\infty$. Therefore for any $E$ of finite measure we have
$$
\mm(E)=\int \nchi_E \,\d\mm=\lim_{t\downarrow 0}\int \nchi_E\rho_t \,\d\mm=\lim_{t\downarrow 0}\int \nchi_E\, \d(F_t)_\ast \mm=\lim_{t\downarrow0}(F_t)_ \ast \mm(E).
$$
This proves that if $\mm(E)>0$, then for $t$ sufficiently small it holds $(F_t)_\ast \mm(E)>0$ as well. Then the conclusion follows from the definition of $\mu_{\alpha,\beta}$.

Thus from \eqref{eq:minmu} it follows   that  there exists and is $\mm\restr{\Omega}$-a.e.\ uniquely determined a Borel function $H$ such that
\[
H=\bar F_{\alpha,\beta}\quad\mu_{\alpha,\beta}-a.e.\qquad\forall \alpha>\beta>0.
\]
We claim that such $H$ has the required properties.  We start by proving that $H\in L^p(\Omega)$ and to this aim we start noticing that
\begin{equation}
\label{eq:limma}
\lim_{\alpha\downarrow0} \frac1\alpha\int_0^\alpha\int_{C_\alpha}f\circ\Fl^Z_t\,\d\mm\,\d t=\int_\Omega f \,\d\mm\qquad\forall f:\Omega\to\R^+\ \textrm{ Borel}.
\end{equation}
This can be easily proved for $f$ bounded and Lipschitz, then the  case of $f\in L^1(\Omega)$ follows by a density argument based on the bound  $(\Fl^Z_t)_*\mm\leq e^{t\|\div Z\|_{L^\infty}}\mm$ and finally the case of non-negative $f$'s comes by monotone approximation.

Now notice  that  by construction (and \eqref{eq:derint}) it holds $\int_{C_\alpha}\mathcal E_{p,[0,\alpha-\beta]}(u\circ\Fl^Z_\cdot) \,\d\mm=\int_{C_\alpha}\int_0^{\alpha-\beta}|H|^p\circ\Fl^Z_t\,\d t\,\d\mm$, hence from the second in \eqref{eq:claimi-iii} we obtain
\begin{equation}
\label{eq:Hen}
\begin{split}
\int_\Omega|H|^p\,\d\mm\stackrel{\eqref{eq:limma}}=&\lim_{\alpha\downarrow0} \frac1\alpha\int_0^\alpha\int_{C_\alpha}|H|^p\circ\Fl^Z_t\,\d\mm\,\d t\\
\stackrel{\phantom{\eqref{eq:limma}}}=&\lim_{\alpha\downarrow0}\lim_{\beta\downarrow0}\frac1{\alpha-\beta}\int_{C_\alpha}\int_0^{\alpha-\beta}|H|^p\circ\Fl^Z_t\,\d t\,\d\mm\\
\stackrel{\phantom{\eqref{eq:limma}}}=&\lim_{\alpha\downarrow0}\lim_{\beta\downarrow0}\frac1{\alpha-\beta}\int_{C_\alpha}\mathcal E_{p,[0,\alpha-\beta]}(u\circ\Fl^Z_\cdot) \,\d\mm\leq E^{Z}_{p}[u,\Omega]\\
\end{split}
\end{equation}
Now we prove \eqref{eq:g3}. Letting $\beta\downarrow0$ in \eqref{eq:acae} we see that for every $\alpha>0$ it holds: $\mm$-a.e.\ $x\in C_\alpha$ the curve $t\mapsto u(\Fl^Z_t(x))$ belongs to $W^{1,p}([0,\alpha],\Y)$ (e.g.\ by recalling the relation between Sobolev and AC curves stated in Theorem \ref{thm:sobaccurve}) and, by definition, its distributional derivative is given by $H\circ\Fl^Z_t$. Thus for $\mm$-a.e.\ $x\in\Omega$ we have that: for every $\alpha\in\Q$ with $\alpha<T_x$ the curve $t\mapsto u(\Fl^Z_t(x))$ belongs to $W^{1,p}([0,\alpha],\Y)$ and its distributional derivative is given by $H\circ\Fl^Z_t$.  Arguing as before by calling into play Theorem \ref{thm:sobaccurve} we conclude that for $\mm$-a.e.\ $x$ the curve $t\mapsto u(\Fl^Z_t(x))$ belongs to $W^{1,p}([0,T_x],\Y)$ and its distributional derivative is given by $H\circ\Fl^Z_t$, as desired.

\noindent{\bf Last statements} The fact that the choice $G:=H$ is the least for which \eqref{eq:g2} holds is a direct consequence of the analogous statement in Lemma \ref{le:basequ2}.  Inequality $\geq$ in \eqref{eq:energy} is proved in \eqref{eq:Hen} while the opposite comes with the proofs  $(iii)\Rightarrow(ii)$ and $(ii)\Rightarrow(i)$.

It remains to prove $L^p(\Omega)$-convergence of $(e^{Z}_{p,\varepsilon}[u,\Omega])^{1/p}$ to $H$. Extend $H$ to the whole $\X$ by putting it 0 outside $\Omega$ and  notice that what we already proved gives
\begin{equation}
\label{eq:ezabove}
(e^{Z}_{p,\varepsilon}[u,\Omega])^{1/p}\leq\frac1\eps\int_0^\eps H\circ\Fl^Z_t\,\d t\quad\to\quad H\qquad\text{ in $L^p(\Omega)$},
\end{equation}
where the claimed convergence can be proved along the same lines used to show \eqref{eq:limma}. Now notice that for $\alpha>\beta>0$, Lemma \ref{le:basequ2} applied to $\X:=C_\alpha$ and $f_t:=u\circ\Fl^Z_t$ in the interval $[0,\alpha-\beta]$ ensures that $(e^{Z}_{p,\varepsilon}[u,\Omega])^{1/p}\circ\Fl^Z\to H\circ\Fl^Z$ in $L^p(C\times[0,\alpha-\beta])$. This is the same as to say that $(e^{Z}_{p,\varepsilon}[u,\Omega])^{1/p}\to H$ in $L^p(\mu_{\alpha,\beta})$ and in particular any $L^p(\Omega)$-weak limit of $(e^{Z}_{p,\varepsilon}[u,\Omega])^{1/p}$ must coincide with $H$ $\mu_{\alpha,\beta}$-a.e.. Thus by \eqref{eq:mumm} we deduce that $(e^{Z}_{p,\varepsilon}[u,\Omega])^{1/p}\weakto H$ in $L^p(\Omega)$, which together with \eqref{eq:ezabove} gives the conclusion.
\end{proof}
Theorem \ref{thm:baseks} and its proof suggest the following definition:
\begin{definition}[The quantity $|\d u(Z)|$]\label{def:mwug} Let $p\in(1,\infty)$, $Z$ a regular vector field on $\X$, $\Omega\subset \X$ open and $u\in \KS {Z}(\Omega,\Y_{\bar y})$. We shall denote by $|\d u(Z)|\in L^p(\Omega)$ the function $H$ given by point $(iv)$ of Theorem \ref{thm:baseks} and appearing in  \eqref{eq:g3}.
\end{definition}
In the smooth category, the quantity $|\d u(Z)|$ is the norm of the differential of $u$ applied to $Z$, whence the notation chosen. Notice that for the moment we only defined $|\d u(Z)|$, not the underlying object $\d u(Z)$, so the notation chosen is purely formal. We will define $\d u(Z)$ in Section \ref{se:diff}.

\bigskip

We conclude this section with the following kind of regularity result which will be useful in what comes next.
\begin{proposition}\label{prop:regalongfl}
Let $p\in(1,\infty)$, $Z$ a regular vector field on $\X$, $\Omega\subset \X$ open and $u\in \KS Z(\Omega,\Y_{\bar y})$.  Then for every $C\subset \Omega$ closed for which $T_C>0$ (recall the definition \eqref{eq:deftc}) and $f\in\Lip_{bs}(\Y)$, the map $[0,T_C)\ni t\mapsto f\circ u\circ\Fl^Z_t\in L^p(C)$ is $C^1$ and for its derivative we have for every $t \in [0,T_{C})$
\begin{equation}
\label{eq:der0}
\Big|\frac\d{\d t}f\circ u\circ\Fl^Z_t\Big|\leq\big(\lip(f)\circ u\,|\d u(Z)|\big)\circ\Fl^Z_t \qquad\mm-a.e.\ on\ C.
\end{equation}
\end{proposition}
\begin{proof} For any $t,s\in[0,T_C)$ we have
\[
|f\circ u\circ\Fl^Z_s-f\circ u\circ\Fl^Z_t|\leq \Lip(f)\,\sfd_\Y(u\circ\Fl^Z_s, u\circ\Fl^Z_t)\quad\mm-a.e.\ on\ C
\]
and thus  \eqref{eq:g2} yields that $t\mapsto f\circ u\circ\Fl^Z_t\in L^p(C)$ is Lipschitz. Since $L^p(C)$ has the Radon-Nikodym property, we deduce that such curve is differentiable for a.e.\ $t$. Then from the identity $\frac{f\circ u\circ\Fl^Z_{s+h}-f\circ u\circ\Fl^Z_{s}}{h}=\frac{f\circ u\circ\Fl^Z_{t+h}-f\circ u\circ\Fl^Z_{t}}{h}\circ\Fl^Z_{s-t}$ we deduce that
\[
(f\circ u\circ\Fl^Z_\cdot)'_s=(f\circ u\circ\Fl^Z_\cdot)'_t\circ\Fl^Z_{s-t}
\]
for every differentiability points $s>t$. Since Proposition \ref{le:contlpfl} grants continuity in $s$ with values in $L^p(C)$ of the right hand side, $C^1$ regularity follows. Then the bound \eqref{eq:der0} follows from the definition of $|\d u(Z)|$, Corollary \ref{cor:chain}  and Lemma \ref{le:basequ}.
\end{proof}

\subsection{Triangle inequality}

Aim of this section is to prove that under suitable natural assumptions it holds the following sort of triangle inequality:
\[
|\d u(\alpha_1Z_1+\alpha_2Z_2)|\leq|\alpha_1|\,|\d u(Z_1)|+|\alpha_2|\,|\d u(Z_2)|.
\]
The study of the above will be divided in two parts: a first (easy) one where we study the effect of multiplication of vector fields by constants and a second (more delicate) where we study sums of vector fields.

\bigskip

We start with the following simple lemma:
\begin{lemma}\label{le:scalato}
Let $(\X,\sfd,\mm)$ be a $\RCD(K,\infty)$ space and $Z$ a regular vector field. Then for every $\alpha,t\geq 0$ we have $\Fl^{\alpha Z}_t=\Fl^Z_{\alpha t}$ $\mm$-a.e.. If $-Z$ is also a regular vector field (i.e.\ if $\div Z\in L^\infty$), the same conclusion holds for any $\alpha\in \R$.
\end{lemma}
\begin{proof}
Start noticing that $\alpha Z$ is also a regular vector field, so that the statement makes sense. To conclude, according to Theorem \ref{thm:AT} it is sufficient to show that if $t\mapsto \rho_t$ solves the continuity equation for $v_t\equiv Z$, then $t\mapsto \rho_{\alpha t}$ solves the continuity equation for $v_t\equiv \alpha Z$. But this is obvious, whence the conclusion follows.
\end{proof}
As a direct consequence of the above we obtain:
\begin{proposition}[Multiplication of the vector field by a constant]\label{prop:mult}
Let $K\in\R$, $(\X,\sfd,\mm)$ be $\RCD(K,\infty)$ space, $\Omega\subset\X$ open and $Z$ a  regular vector field on it. Let $(\Y,\sfd_\Y)$ be a complete metric space and $u\in \KS {Z}(\Omega,\Y_{\bar y})$.

Then for every $\alpha\geq 0$ we also have $u\in \KS {\alpha Z}(\Omega,\Y_{\bar y})$ and $|\d u(\alpha Z)|=|\alpha||\d u(Z)|$. If $-Z$ is also a regular vector field,  the same conclusion holds for any $\alpha\in \R$.

\end{proposition}
\begin{proof} It is clear that if $t\mapsto\gamma_t$ is absolutely continuous then so is $t\mapsto\gamma_{\alpha t}$ and with metric speed which changes by a factor $|\alpha|$. Then by Theorem \ref{thm:sobaccurve} the same holds for Sobolev curves and distributional derivatives. Then conclusion easily follows from Theorem \ref{thm:baseks}.
\end{proof}
We now turn to the study of the effect of the sum of vector fields on Regular Lagrangian Flows and start with a simple result about stability of convergence  in measure under left composition:
\begin{lemma}\label{le:compleft} Let $T_n:\X\to\X$, $n\in\N\cup\{\infty\}$ be Borel and such that $T_n\to T_\infty$ locally in measure as $n\to\infty$. Assume also that the measures $(T_n)_*\mm$ are locally equi-absolutely continuous w.r.t.\ $\mm$, i.e.\ that: for every $\eps>0$ and $B\subset\X$ bounded there is $\delta>0$ such that for every $E\subset B$ Borel with $\mm(E)<\delta$ we have $(T_n)_*\mm(E)\leq\eps$ for every $n\in\N$.

Then for every complete metric space $\Y$ and every Borel map $u:\X\to \Y$ which is essentially separably valued we have that $(u\circ T_n)$ converges locally in measure to $u\circ T_\infty$.
\end{lemma}
\begin{proof} Replacing $\Y$ with a closed separable subset containing, up to negligible sets, the image of $u$ we can assume that $\Y$ is separable. Then up to a  (Kuratowski) isometric embedding of $\Y$ in $\ell^\infty$ we can replace the former with the latter. Then Lemma \ref{le:cbdense} and a simple cut-off argument shows that $C_b(\X,\ell^\infty)$ is dense in the space of essentially separably valued Borel maps from $\X$ to $\ell^\infty$ w.r.t.\ to local convergence in measure.

Now let $\mm'\in\mathscr P(\X)$ be such that $\mm\ll\mm'\ll\mm$ and notice that the assumption on equi-absolute continuity of  $(T_n)_*\mm$  implies
\begin{equation}
\label{eq:equiacp}
\text{\rm $\forall \eps>0 \ \exists\delta>0\ $ s.t.\   $\forall E\subset \X$ Borel the bound $\mm'(E)\leq \delta$ implies $(T_n)_*\mm'(E)\leq\eps$\ $\forall n\in\N\cup\{\infty\}$}
\end{equation}
and recall that the distance $\sfd_0(u,v):=\int1\wedge\sfd_{\ell^\infty}(u,v)\,\d\mm'$ metrizes the  local convergence in measure.

Then for $u:\X\to\ell^\infty$ Borel and essentially separably valued and  $\eps>0$ let first $\delta$ be given by \eqref{eq:equiacp} and then $v\in C_b(\X,\ell^\infty)$ be such that for $E:=\{\sfd_{\ell^\infty}(u,v)>\eps\}$ it holds $\mm'(E)<\delta$. We have
\begin{equation}
\label{eq:tr}
\sfd_0(u\circ T_\infty,u\circ T_n)\leq\sfd_0(u\circ T_\infty,v\circ T_\infty)+\sfd_0(v\circ T,v\circ T_n)+\sfd_0(u\circ T_n,v\circ T_n)\qquad\forall n\in\N
\end{equation}
and for every $n\in\N\cup\{\infty\}$ it holds
\[
\begin{split}
\sfd_0(u\circ T_n,v\circ T_n)&= \int_E1\wedge\sfd_{\ell^\infty}(u,v)\,\d(T_n)_*\mm'+\int_{\X\setminus E}1\wedge\sfd_{\ell^\infty}(u,v)\,\d(T_n)_*\mm'\stackrel{\eqref{eq:equiacp}}\leq\eps+\eps.
\end{split}
\]
Since the  continuity of $v$ and the dominated convergence theorem give that $\sfd_0(v\circ T,v\circ T_n)\to0$ as $n\to\infty$, from \eqref{eq:tr} we obtain
\[
\lims_{n\to\infty}\sfd_0(u\circ T_\infty,u\circ T_n)\leq 4\eps
\]
and by the arbitrariness of $\eps>0$ we conclude.
\end{proof}
The core of the matter for what concerns the triangle inequality is the following lemma: here we make crucial use of the stability results for Regular Lagrangian Flows that we obtained in Section \ref{se:rcdrlf}.
\begin{lemma}\label{le:sum}
Let $K\in\R$, $(\X,\sfd,\mm)$ be $\RCD(K,\infty)$ space, $\Omega\subset\X$ open and $Z_1,Z_2$ two regular vector fields on it. Let $(\Y,\sfd_\Y,{\bar y})$ be a pointed complete  space and $u\in \KS {Z_1}(\Omega,\Y_{\bar y})\cap \KS {Z_2}(\Omega,\Y_{\bar y})$. For $C\subset \Omega$ closed put $T_C:=\frac12\min\{T^{1}_{C},T^{2}_{C}\}$, where $T^{j}_{C}$, $j=1,2$ is defined as in \eqref{eq:deftc} for the vector field $Z_j$, $j=1,2$ and the set $C$.

Then for every $f\in \Lip_{bs}(\Y)$ we have that the map $[0,T_C)\ni t\mapsto f\circ u\circ  \Fl^{Z_1+Z_2}_t\in L^p(C)$ is $C^1$ and for its derivative at time 0 we have
\begin{equation}
\label{eq:dersum0}
\frac{\d}{\d t}\big(f\circ u\circ  \Fl^{Z_1+Z_2}_t\big)\restr{t=0}=\frac{\d}{\d t}\big(f\circ u\circ  \Fl^{Z_1}_t\big)\restr{t=0}+\frac{\d}{\d t}\big(f\circ u\circ  \Fl^{Z_2}_t\big)\restr{t=0}\,.
\end{equation}
\end{lemma}
\begin{proof} Let $\rho_0$ be a bounded probability density with support in $C$ and $T>0$. Define $Z_n$ as in Proposition \ref{prop:sumvect} and let $(\Fl^+_t),(\Fl^n_t),(\Fl^1_t),(\Fl^2_t)$ be the Regular Lagrangian Flows of $Z_1+Z_2,Z_n,Z_1,Z_2$ respectively and $(\rho^+_t),(\rho^n_t),(\rho^1_t),(\rho^2_t)$ the corresponding solutions of the continuity equation starting from $\rho_0$.

We know from \eqref{eq:bcompr} and our assumptions that these Regular flows have locally uniformly bounded compression in $t\in[0,T_C)$, thus from the stability result Theorem \ref{thm:stabfl} (coupled with Proposition \ref{prop:sumvect}) and Lemma \ref{le:compleft} above we deduce that $(f\circ u\circ\Fl^n_t)$ converges in measure to $(f\circ u\circ\Fl^+_t)$ as $n\to\infty$ for any $t>0$. Since these functions are uniformly bounded (because $f$ is bounded) and $\rho_0$ has bounded support, we deduce that
\begin{equation}
\label{eq:forlim}
\int \big(f\circ u\circ\Fl^+_{t_1}-f\circ u\circ\Fl^+_{t_0}\big)\,\rho_0\,\d\mm=\lim_{n\to\infty}\int\big( f\circ u\circ\Fl^n_{t_1}-f\circ u\circ\Fl^n_{t_0}\big)\,\rho_0\,\d\mm\qquad\forall t_1\geq t_0\geq 0.
\end{equation}
Fix $t_1\geq t_0\geq 0$ and for $n\in\N$ let $I_1({n}),I_0({n})\in \N$ be such that $t_j\in[\frac{2I_j(n)}{2^n},\frac{2I_j(n)+2}{2^n})$, $j=0,1$, and use the very definition of $\Fl^n_t$ and the regularity property stated in Proposition \ref{prop:regalongfl} to write
\begin{equation}
\label{eq:nsplit}
\begin{split}
\int \big(f\circ u\circ\Fl^n_{t_1}&-f\circ u\circ\Fl^n_{t_0}\big)\,\rho_0\,\d \mm\\
&=\iint_\frac{2I_1(n)}{2^n}^{t_1} \frac\d{\d t}\big(f\circ u\circ \Fl^{n}_{t}\big)\,\rho_0\,\d t\,\d\mm-\iint_\frac{2I_0(n)}{2^n}^{t_0} \frac\d{\d t}\big(f\circ u\circ \Fl^{n}_{t}\big)\,\rho_0\,\d t\, \d\mm\\
&\qquad+\sum_{i=I_0(n)}^{I_1(n)-1}\underline{\iint_{\frac{2i}{2^n}}^{\frac{2i+1}{2^n}}\frac{\d}{\d t}\big(f\circ u\circ\Fl^{2 Z_1}_{t-\frac{2i}{2^n}}\circ\Fl^n_{\frac{2i}{2^n}}\big)\,\rho_0\,\d t\, \d\mm}\\
&\qquad\qquad+\iint_{\frac{2i+1}{2^n}}^{\frac{2i+2}{2^n}}\frac\d{\d t}\big(f\circ u\circ \Fl^{2 Z_2}_{t-\frac{2i+1}{2^n}}\circ\Fl^n_{\frac{2i+1}{2^n}}\big)\,\rho_0\,\d t\, \d\mm.
\end{split}
\end{equation}
We shall study the underlined term in the above. It is easy to verify that the uniqueness of Regular Lagrangian Flows yields $\Fl^{2Z_1}_t=\Fl^{Z_1}_{2t}$ for every $t\geq 0$, hence recalling the definition of $\rho^1_{n,t}$ given in Proposition \ref{prop:pertr} we have
\[
\begin{split}
\iint_{\frac{2i}{2^n}}^{\frac{2i+1}{2^n}}\frac{\d}{\d t}\big(f\circ u\circ\Fl^{2 Z_1}_{t-\frac{2i}{2^n}}\circ\Fl^n_{\frac{2i}{2^n}}\big)\,\rho_0\,\d t\,\d\mm&=\iint_{\frac{2i}{2^n}}^{\frac{2i+1}{2^n}}\frac{\d}{\d t}\big(f\circ u\circ\Fl^{ Z_1}_{2(t-\frac{2i}{2^n})}\circ\Fl^n_{\frac{2i}{2^n}}\big)\,\rho_0\,\d t\,\d\mm\\
&=2\iint_{\frac{2i}{2^n}}^{\frac{2i+1}{2^n}}\frac{\d}{\d s}\big(f\circ u\circ\Fl^{Z_1}_s\big)\restr{s=0}\circ \Fl^{ Z_1}_{2(t-\frac{2i}{2^n})}\circ\Fl^n_{\frac{2i}{2^n}}\,\rho_0\,\d t\, \d\mm\\
&=\iint_{\frac{2i}{2^n}}^{\frac{2(i+1)}{2^n}}\frac{\d}{\d s}\big(f\circ u\circ\Fl^{Z_1}_s\big)\restr{s=0}\rho^1_{n,t}\,\d t\,\d\mm.
\end{split}
\]
The convergence properties stated in Proposition \ref{prop:pertr} ensure that
\[
\int\frac{\d}{\d s}\big(f\circ u\circ\Fl^{Z_1}_s\big)\restr{s=0}\rho^1_{n,t}\,\d\mm\quad\to\quad\int\frac{\d}{\d s}\big(f\circ u\circ\Fl^{Z_1}_s\big)\restr{s=0}\rho^+_{t}\,\d\mm\qquad\forall t\in[0,T_C).
\]
Handling analogously the other two terms in the right-hand-side of \eqref{eq:nsplit} and recalling \eqref{eq:forlim}, by the dominated convergence theorem (recall that the densities $\rho^n_t$ are uniformly bounded) we obtain
\[
\int\big( f\circ u\circ\Fl^+_{t_1}-f\circ u\circ\Fl^+_{t_0}\big)\,\rho_0\,\d\mm=\iint_{t_0}^{t_1}\Big(\frac{\d}{\d s}\big(f\circ u\circ\Fl^{Z_1}_s\big)\restr{s=0}+\frac{\d}{\d s}\big(f\circ u\circ\Fl^{Z_2}_s\big)\restr{s=0}\Big)\circ\Fl^+_t\,\rho_0\,\d t\,\d\mm.
\]
Then the arbitrariness of $\rho_0$ forces
\[
 f\circ u\circ\Fl^+_{t_1}-f\circ u\circ\Fl^+_{t_0} =\int_{t_0}^{t_1}\Big(\frac{\d}{\d s}\big(f\circ u\circ\Fl^{Z_1}_s\big)\restr{s=0}+\frac{\d}{\d s}\big(f\circ u\circ\Fl^{Z_2}_s\big)\restr{s=0}\Big)\circ\Fl^+_t\,\d t,
\]
which (recalling the bound \eqref{eq:der0}) shows that $t\mapsto f\circ u\circ\Fl^+_{t}\in L^p(C)$ is Lipschitz. We then conclude as for Proposition \ref{prop:regalongfl}: since $L^p(C)$ has the Radon-Nikodym property we deduce that the formula
\[
\frac{\d }{\d t}f\circ u\circ\Fl^+_{t}=\Big(\frac{\d}{\d s}\big(f\circ u\circ\Fl^{Z_1}_s\big)\restr{s=0}+\frac{\d}{\d s}\big(f\circ u\circ\Fl^{Z_2}_s\big)\restr{s=0}\Big)\circ\Fl^+_t
\]
holds for a.e.\ $t\in[0,T_C)$, then we use the semigroup property of $(\Fl^+_t)$, we recall the continuity property stated in Proposition \ref{le:contlpfl} to obtain the  desired $C^1$ regularity and the validity of \eqref{eq:dersum0}.
\end{proof}
The main result of this section is now easily obtainable:
\begin{proposition}\label{prop:triangle}
Let $K\in\R$, $(\X,\sfd,\mm)$ be $\RCD(K,\infty)$ space, $\Omega\subset\X$ open and $Z_1,Z_2$ two regular vector fields on it. Let $(\Y,\sfd_\Y,{\bar y})$ be a pointed complete space and $u\in \KS {Z_1}(\Omega,\Y_{\bar y})\cap \KS {Z_2}(\Omega,\Y_{\bar y})$.

Then $u\in \KS {Z_1+Z_2}(\Omega,\Y_{\bar y})$ and
\begin{equation}
\label{eq:triangle}
|\d u(Z_1+Z_2)|\leq |\d u(Z_1)|+|\d u(Z_2)|\qquad\mm-a.e..
\end{equation}
\end{proposition}
\begin{proof} By Lemma \ref{le:sum} and inequality \eqref{eq:der0}  we know that for any $C\subset\Omega$ closed, $f\in\Lip_{bs}(\Y)$ and $t,s\in[0,T_C)$, $t<s$ it holds
\[
|f\circ u\circ  \Fl^{Z_1+Z_2}_s-f\circ u\circ  \Fl^{Z_1+Z_2}_t|\leq \Lip(f)\int_t^s\big(|\d u(Z_1)|+|\d u(Z_2)|\big)\circ  \Fl^{Z_1+Z_2}_r\,\d r\quad\mm-a.e.\ on\ C.
\]
Taking the supremum as  $f$ runs on the countable set given by Lemma \ref{le:spcount}, by \eqref{eq:distsup} we get
\[
\sfd_\Y(u\circ  \Fl^{Z_1+Z_2}_s,u\circ  \Fl^{Z_1+Z_2}_t)\leq\int_t^s\big(|\d u(Z_1)|+|\d u(Z_2)|\big)\circ  \Fl^{Z_1+Z_2}_r\,\d r\quad\mm-a.e.\ on\ C
\]
and since $|\d u(Z_1)|+|\d u(Z_2)|\in L^p(\Omega)$, according to Theorem \ref{thm:baseks} this is sufficient to conclude.
\end{proof}

\subsection{The differential $\d u(Z)$}\label{se:diff}

In this section we show that `behind' the definition of $|\d u(Z)|$ there is an actual object which we can rightfully call 'differential of $u$ in the direction of $Z$' whose norm coincides a.e.\ with $|\d u(Z)|$. The kind of construction that we use is strongly reminiscent of  - and deeply motivated by -  the one proposed in \cite{GPS18}; the difference here is that our function $u$  is differentiable only in the direction of $Z$. For this reason we won't define the differential $\d u$ of $u$, but only its action on $Z$.

\bigskip

Throughout this section we fix $p\in(1,\infty)$, a regular vector field $Z$  on the $\RCD(K,\infty)$ space $\X$, an open set $\Omega\subset\X$ and  a map $u\in \KS Z(\Omega,\Y_{\bar y})$. Moreover, we shall assume henceforth that $\Y$ is separable (this is helpful - albeit not really necessary - because we shall consider measures on $\Y$ and then Sobolev functions on it, whose theory usually assumes separability).

We then put
\[
\mu:=u_*(|\d u(Z)|^p\mm\restr\Omega)
\]
and notice that $\mu$ is a finite Radon measure on $\Y$. We shall consider Sobolev functions on $(\Y,\sfd_\Y,\mu)$ and to emphasize the choice of the measure (which can occasionally be different from $\mu$) we shall denote the corresponding notion of cotangent module and differential by $L^0_\mu(T^*\Y)$ and $|\d_\mu f|$.

Notice that if  $f:\Y\to\R$ is defined up to $\mu$-a.e.\ equality, then $f\circ u$ is defined up to $\mm$-a.e.\ equality on $\{|\d u(Z)|>0\}$; hence the function $f\circ u|\d u(Z)|$ is defined up to $\mm$-a.e.\ equality on $\Omega$. Then the trivial identity
\[
\int_\Omega \big|f\circ u|\d u(Z)| \big|^p\,\d\mm=\int |f|^p\circ u\,\d (|\d u(Z)|^p\mm\restr\Omega)=\int |f|^p\,\d\mu
\]
shows that
\begin{equation}
\label{eq:pbnew}
\text{the map $\qquad L^p(\mu)\ni f\quad\mapsto\quad f\circ u|\d u(Z)| \in L^p(\mm\restr\Omega)\qquad $ is linear and continuous.}
\end{equation}
The relation between Sobolev functions on $(\Y,\sfd_\Y,\mu)$ and their pullback obtained by right composition with $u$ is described in the next proposition:
\begin{proposition}\label{prop:chain}  Let $p\in(1,\infty)$, $Z$ a regular vector field on $\X$, $\Omega\subset \X$ open, $u\in \KS Z (\Omega,\Y_{\bar y})$ and $\mu:=u_*(|\d u(Z)|^p\mm)$.

For every $f\in L^\infty\cap W^{1,p}(\Y,\sfd_\Y,\mu)$ there exists $g\in L^0(\mm\restr\Omega)$ which coincides with  $f\circ u$ $\mm$-a.e.\ on $\{|\d u(Z)|>0\}$ and for which the following holds: for any $C\subset \Omega$ closed with $T_C>0$,  the map $t\mapsto (g\circ\Fl^Z_t-g)$ belongs to $C^1([0,T_C),L^p(C))$ and for its derivative at time 0 it holds the bound
 \begin{equation}
\label{eq:boundder0}
\Big|\frac\d{\d t}\big(g\circ\Fl^Z_t\big)\restr{t=0}\Big|\leq |\d_\mu f|\circ u |\d u(Z)|\qquad\mm-a.e.\ \textrm{on }C.
\end{equation}
Moreover, if $\tilde g\in L^0(\mm\restr\Omega)$ is another function which coincides with  $f\circ u$ $\mm$-a.e.\ on $\{|\d u(Z)|>0\}$ and
such that $t\mapsto g\circ\Fl^Z_t$ belongs to $C^1([0,T_C),L^p(C))$ and so that
\begin{equation}
\label{eq:compg}
\Big|\frac\d{\d t}\big(\tilde g\circ\Fl^Z_t\big)\restr{t=0}\Big|\leq h\circ u|\d u(Z)|\qquad\mm-a.e.\ \textrm{on }C\qquad\text{ for some }h\in L^p(\mu),
\end{equation}
then
\[
\frac\d{\d t}\big(g\circ\Fl^Z_t\big)\restr{t=0}=\frac\d{\d t}\big(\tilde g\circ\Fl^Z_t\big)\restr{t=0}.
\]
\end{proposition}
\begin{proof} Let   $(\tilde f_n)\subset \LIP_{bs}(\Y)$ be an optimal sequence as in \eqref{eq:optlip}. Observe that since $f$ is bounded, by truncation we can assume the $\tilde f_n$'s to be uniformly bounded. Thus the functions $\tilde f_n\circ u$ are also uniformly bounded and hence up to pass to a subsequence - not relabeled - they converge to some limit function $g$ weakly in $L^p(\Omega)$. By Mazur's lemma, there is a sequence $(f_n)$ of convex combinations of the $\tilde f_n$'s such that $f_n\circ u\to g$ strongly in $L^p(\Omega)$ and it is easy to verify that $(f_n)$ is still optimal for $f$ as in \eqref{eq:optlip}.

Now notice that from Proposition \ref{prop:regalongfl} we know that for every $t,s\in[0,T_C)$,  $t<s$ it holds
\[
|f_n\circ u\circ \Fl^Z_s-f_n\circ u\circ \Fl^Z_t|\leq \int_t^s \big(\lip(f_n)\circ u |\d u(Z)|\big)\circ\Fl^Z_r\,\d r\qquad\mm-a.e.\ on\ C.
\]
The fact that $(\Fl^Z_t)$ has bounded compression and the construction give that the left hand side in the above converges to $|g\circ\Fl^Z_s-g\circ\Fl^Z_t|$ in  $L^p(C)$, while the same fact and \eqref{eq:pbnew} give that $(\lip(f_n)\circ u |\d u(Z)|)\circ\Fl^Z_r\to (|\d_\mu f|\circ u|\d u(Z)|)\circ\Fl^Z_r$ in $L^p(C)$ uniformly on $r\in[t,s]$. Thus passing to the limit in the above we obtain
\begin{equation}
\label{eq:forder0}
|g\circ\Fl^Z_s-g\circ\Fl^Z_t|\leq \int_t^s\big( |\d_\mu f|\circ u\,|\d u(Z)|\big)\circ\Fl^Z_r\,\d r\qquad\mm-a.e.\ on\ C
\end{equation}
which in turn gives local Lipschitz regularity for $t\mapsto g\circ\Fl^Z_t\in L^p(C)$. Then $C^1$ regularity follows, as in Proposition \ref{prop:regalongfl}, from the semigroup property of $(\Fl^Z_t)$, which implies that
\[
\frac\d{\d t}(g\circ\Fl^Z_t)\restr{t=t_1}=\frac\d{\d t}(g\circ\Fl^Z_t)\restr{t=t_0}\circ\Fl^Z_{t_1-t_0}
\]
for any two differentiability points $t_0,t_1\in [0,T_C)$, and the continuity statement Proposition \ref{le:contlpfl} which grants that the derivative of $g\circ\Fl^Z_t$ has a continuous representative. The bound \eqref{eq:boundder0} is then a direct consequence of \eqref{eq:forder0}.

For the second part of the statement, we start noticing that \eqref{eq:boundder0} and \eqref{eq:compg} grant that $\frac\d{\d t}\big(g\circ\Fl^Z_t\big)\restr{t=0}=\frac\d{\d t}\big(\tilde g\circ\Fl^Z_t\big)\restr{t=0}=0$ $\mm$-a.e.\ on $\{|\d u(Z)|=0\}$, so to conclude we need to prove that these two derivatives coincide on $\{|\d u(Z)|>0\}$.

 By the properties of $g,\tilde g$ and Lemma \ref{le:basequ}  we know that for $\mm$-a.e.\ $x\in C$ the map $[0,T_C)\ni t\mapsto g(\Fl^Z_t(x))$ belongs to $W^{1,p}_{loc}([0,T_C)$ and similarly for  $t\mapsto \tilde g(\Fl^Z_t(x))$. For any such $x$, by the locality property of the distributional differential we see that $\partial_tg(\Fl^Z_t(x))=\partial_t \tilde g(\Fl^Z_t(x))$ for a.e.\ $t\in\{s:g(\Fl^Z_s(x))=\tilde g(\Fl^Z_s(x))\}$.

Using again Lemma \ref{le:basequ} we see that $\frac{\d}{\d t}(g\circ\Fl^Z_t)(x)=\frac{\d}{\d t}(\tilde g\circ\Fl^Z_t)(x)$ $\mm\times\mathcal L^1$-a.e.\ $(x,t)\in C\times[0,T_C)$ such that $\Fl^Z_t(x)\in\{g=\tilde g\}$. Then since $g=f\circ u=\tilde g$ on $\{|\d u(Z)|>0\}$, using the $C^1$ regularity of $t\mapsto g\circ \Fl^Z_t, \tilde g\circ \Fl^Z_t\in L^p(C)$ we conclude that $\frac{\d}{\d t}(g\circ\Fl^Z_t)\restr{t=0}(x)=\frac{\d}{\d t}(\tilde g\circ\Fl^Z_t)\restr{t=0}(x)$ $\mm$-a.e.\ $x\in\{|\d u(Z)|>0\}$, as desired.
\end{proof}
Ideally, $\d u(Z)$ should be defined as the element of the dual of the pullback of $L^0_\mu(T^*\Y)$ via $u$ characterized by its action on $[u^*\d_\mu f]$ via the formula
\[
[u^*\d_\mu f](\d u(Z))=\frac\d{\d t}( g\circ\Fl^Z_t)\restr{t=0}\qquad a.e.\ on\ C
\]
for any $C,g$ as in Proposition \ref{prop:chain}. From the technical point of view, the above approach has the problem that it does not really define any object $\mm$-a.e., but only $\mm\restr{\{|\d u(Z)|>0\}}$-a.e.\ (notice that both the pullback $u^*L^0_\mu(T^*\Y)$ and its dual are not $L^0(\mm\restr\Omega)$-normed modules, but  in fact $L^0(\mm\restr{\{|\d u(Z)|>0\}})$-normed). This is not a crucial issue, because it is natural to impose that $\d u(Z)$ is 0 outside the set $\{|\d u(Z)|>0\}$, but needs to be taken care of:  we shall proceed as in \cite{GPS18} by using the extension functor discussed in Section \ref{se:sobcal}.

\begin{definition}[The object $\d u(Z)$] \label{def:duz}
 Let $p\in(1,\infty)$, $Z$ a regular vector field on $\X$, $\Omega\subset \X$ open, $u\in \KS Z(\Omega,\Y)$ and $\mu:=u_*(|\d u(Z)|^p\mm)$. Then $\d u(Z)$ is the element of  $({\rm Ext}(u^*L^0_\mu(T^*\Y)))^*$ characterized by: for any $f\in L^\infty\cap W^{1,p}(\Y,\sfd_\Y,\mu)$ and $C\subset\Omega $ with $T_C>0$ it holds
\begin{equation}
\label{eq:defduz}
\big({\rm ext}([u^*\d_{\mu} f])(\d u(Z))\big)\restr C=\Big(\frac\d{\d t}( g\circ\Fl^Z_t)\restr{t=0}\Big)\restr C,
\end{equation}
where $g$ is related to $f$ and $C$ via Proposition \ref{prop:chain} and the derivative is intended in $L^p(C)$.
\end{definition}
We now prove that the definition is well-posed:
\begin{proposition}\label{prop:basediff}
 Let $p\in(1,\infty)$, $Z$ a regular vector field on $\X$, $\Omega\subset \X$ open, $u\in \KS Z(\Omega,\Y_{\bar y})$. Then the definition of $\d u(Z)$ is well-posed. Moreover, the expression $|\d u (Z)|$ is unambiguous, i.e.\ its value in the sense of Definition \ref{def:mwug} coincides with the pointwise norm of $\d u(Z)$, which is defined as
 \begin{equation}
\label{eq:normduz}
|\d u(Z)|:= \esssup \omega(\d u(Z)),
\end{equation}
where the $\esssup$ is taken among all $\omega\in {\rm Ext}(u^*L^0_\mu(T^*\Y))$ with $|\omega|\leq 1$ $\mm$-a.e.\ on $\Omega$.
\end{proposition}
\begin{proof} The second part of Proposition \ref{prop:chain} ensures that  the right hand side in \eqref{eq:defduz} depends only on $f,C$ and not on the particular choice of $g$ as given by the first part of the same statement.  In particular, for given $f\in L^\infty\cap W^{1,p}(\Y,\sfd_\Y,\mu)$ there is a unique function $T(f)\in L^0(\mm\restr{\{|\d u(Z)|>0\}})$ such that $T(f)\restr C=\frac\d{\d t}( g\circ\Fl^Z_t)\restr{t=0}$ for any $C,g$ as in Proposition \ref{prop:chain}. It is clear that $T(f)$ depends linearly on $f$ and this fact together with  \eqref{eq:boundder0} show that $T(f)=T(f')$ $\mm$-a.e.\ on $u^{-1}(\{\d_\mu f=\d_\mu f'\})$, thus $T$ passes to the quotient and defines a linear operator $\tilde T:\{\d_\mu f:f\in  L^\infty\cap W^{1,p}(\Y,\sfd_\Y,\mu)\}\to L^0(\mm\restr{\{|\d u(Z)|>0\}})$, which satisfies
\[
|\tilde T(\d_\mu f)|\leq |\d_\mu f|\circ u|\d u(Z)|\qquad\mm-a.e.\ on\ \{|\d u(Z)|>0\}.
\]
Then the universal property of the pullback ensures that there is a unique $L^0(\mm\restr{\{|\d u(Z)|>0\}})$-linear and continuous operator $S:u^*(L^0_\mu(T^*\Y))\to L^0(\mm\restr{\{|\d u(Z)|>0\}})$  satisfying $S([u^*\d_\mu f])=\tilde T(\d_{\mu} f)$ for every $f\in W^{1,p}(\Y,\sfd_\Y,\mu)$, i.e.\ such that
\[
S([u^*\d_\mu f])\restr C=\frac\d{\d t}( g\circ\Fl^Z_t)\restr{t=0}
\]
for any $C,g$ as above, and such $S$ satisfies
\begin{equation}
\label{eq:snorm}
|S(\omega )|\leq  |\omega|\,|\d u(Z)|\qquad\mm-a.e.\ on\ \{|\d u(Z)|>0\}\qquad\forall \omega\in u^*(L^0_\mu(T^*\Y)).
\end{equation}
It is then clear that we can uniquely extend $S$ to an $L^0(\mm\restr\Omega)$-linear and continuous operator $\d u(Z)$ from ${\rm Ext}(u^*(L^0_\mu(T^*\Y)))$ to ${\rm Ext}(L^0(\mm\restr{\{|\d u(Z)|>0\}}))\subset L^0(\mm\restr\Omega)$ (i.e.\ to an element of $({\rm Ext}(u^*L^0_\mu(T^*\Y)))^*$) and since we have already showed that $\frac\d{\d t}( g\circ\Fl^Z_t)\restr{t=0}=0$ $\mm$-a.e.\ on $\{|\d u(Z)|=0\}$, we proved existence and uniqueness of $\d u (Z)\in ({\rm Ext}(u^*L^0_\mu(T^*\Y)))^*$ satisfying \eqref{eq:defduz}.

We now turn to the claim about $|\d u(Z)|$ and temporarily denote the quantity defined in \eqref{eq:normduz} by $|\d u(Z)|'$, keeping the notation $|\d u(Z)|$ for the one in Definition \ref{def:mwug}. Notice that by \eqref{eq:snorm} it easily follows that
\[
|\omega(\d u(Z))|\leq   |\omega|\,|\d u(Z)|\qquad\mm-a.e.\ on\ \Omega \qquad\forall \omega\in {\rm Ext}(u^*(L^0_\mu(T^*\Y))),
\]
which in turn implies  $|\d u(Z)|'\leq |\d u(Z)|$.

For the converse inequality, we start claiming that for $f\in \Lip_{bs}(\Y)$ the choice $g:=f\circ u$ in Proposition \ref{prop:chain} is admissible for any $C$. This follows from the second part of Proposition \ref{prop:chain} and the observation that the arguments already used in the proof of  Proposition \ref{prop:chain} show that $t\mapsto g\circ\Fl^Z_t\in L^p(C)$ is $C^1$ with
\[
\big|\frac\d{\d t}( g\circ\Fl^Z_t)\restr{t=0}\big|\leq \lip(f)\circ u\,|\d u(Z)|.
\]
Now let $(f_n)\subset\Lip_{bs}(\Y)$ be given by Lemma \ref{le:spcount}, notice that $|\d_\mu f_n|\leq 1$ $\mu$-a.e.\ for every $n$ and thus
\begin{equation}
\label{eq:normp}
|\d u(Z)|'\geq \sup_n\big({\rm ext}([u^*\d_\mu f_n])\big)(\d u(Z))\qquad\mm-a.e.\ on\ \Omega.
\end{equation}
On the other hand, recalling the Definition \ref{def:mwug}, point $(iv)$ in Theorem \ref{thm:baseks}, Lemma \ref{le:spcount} and the semigroup property of $(\Fl^Z_t)$ we see that for any $C\subset\Omega$ with $T_C>0$ it holds
\[
|\d u(Z)|\circ\Fl^Z_t=\sup_n\frac\d{\d t}(f_n\circ u\circ\Fl^Z_t)=\Big(\sup_n\big({\rm ext}([u^*\d_\mu f_n])\big)(\d u(Z))\Big)\circ\Fl^Z_t
\]
$\mm$-a.e.\ on $C$ for a.e.\ $t\in[0,T_C)$. By the continuity property in Proposition \ref{le:contlpfl} we can pick $t=0$ in the above, so that from \eqref{eq:normp} and the arbitrariness of $C$ we conclude $|\d u(Z)|'\geq |\d u(Z)|$, as desired.
\end{proof}
Albeit we cannot define the differential $\d u$ of $u$, if $u$ belongs to the Korevaar-Schoen spaces for two different vector fields $Z_1,Z_2$, and thus also for $Z_1+Z_2$ as proved in Proposition \ref{prop:triangle}, we expect that  $\d u(Z_1+Z_2)=\d u(Z_1)+\d u(Z_2)$ in some sense. This is indeed true, the precise formulation being:
\begin{proposition}[`Linearity' of the differential]\label{prop:lindif}
 Let $p\in(1,\infty)$, $Z_1,Z_2$ two regular vector fields on $\X$, $\Omega\subset \X$ open, $u\in \KS {Z_1}(\Omega,\Y_{\bar y})\cap \KS {Z_2}(\Omega,\Y_{\bar y})$.
We define $\mu_1,\mu_2,\mu_+$ as $u_*(|\d u(Z_1)|^p\mm\restr\Omega),u_*(|\d u(Z_2)|^p\mm\restr\Omega),u_*(|\d u(Z_1+Z_2)|^p\mm\restr\Omega)$ respectively. Then for every $f\in \Lip_{bs}(\Y)$ and $\alpha_1,\alpha_2\geq 0$ we have
\begin{equation}
\label{eq:lindiff}
{\rm ext}([u^*\d_{\mu_+} f])(\d u(\alpha_1Z_1+\alpha_2Z_2))=\alpha_1{\rm ext}([u^*\d_{\mu_1} f])(\d u(Z_1))+\alpha_2{\rm ext}([u^*\d_{\mu_2} f])(\d u(Z_2)).
\end{equation}
If $-Z_i$ is regular as well, $i=1,2$, then the same conclusion holds for any  $\alpha_i\in\R$.
\end{proposition}
\begin{proof}Direct consequence of Lemma \ref{le:sum}, Lemma \ref{le:scalato} and the fact, already observed in the proof of Proposition \ref{prop:basediff} above, that for $f \in\Lip_{bs}(\Y)$ the choice $g:=f\circ u$ is admissible in the definition of $\d u(Z)$.
\end{proof}
Notice that in principle one could use this last proposition to prove the triangle inequality \eqref{eq:triangle}, but this would provide no real save of time, being also this last statement fully based on the crucial Lemma \ref{le:sum}.

\bigskip

We conclude pointing out a duality formula for the pointwise norm of $\d u(Z)$; we remark that the interesting part of the formula is in the fact that 2 different (ordered) measures come into play when computing the differentials of functions on $\Y$:
\begin{corollary}\label{cor:dualnorm}  Let $p\in(1,\infty)$, $Z$ a regular vector field on $\X$, $\Omega\subset \X$ open, $u\in \KS {Z}(\Omega,\Y_{\bar y})$ and $\mu:=u_*(|\d u(Z)|^p\mm)$. Also, let $w\in L^1(\Omega)$ be such that $w\geq |\d u(Z)|^p$ $\mm$-a.e.\   and put $\mu_w:=u_*(w\mm)$. Then
\begin{equation}
\label{eq:dualduz}
\frac1p|\d u(Z)|^p=\esssup_{f\in\Lip_{bs}(\Y)} \ {\rm ext}([u^*\d_\mu f])(\d u(Z))-\frac1q{\rm ext}\big(|\d_{\mu_w}f|^q\circ u\big ),
\end{equation}
where $\frac1p+\frac1q=1$.
\end{corollary}
\begin{proof}
From the bound \eqref{eq:gradord} we deduce
\[
\frac1q|{\rm ext}([u^*\d_\mu f])|^q=\frac1q{\rm ext}(|\d_\mu f|^q\circ u)\leq\frac1q{\rm ext}(|\d_{\mu_w} f|^q\circ u)
\]
for any $f\in \Lip_{bs}(\Y)$ and thus Young's inequality gives
\[
{\rm ext}([u^*\d_\mu f])(\d u(Z))\leq |\d u(Z)|\,|{\rm ext}([u^*\d_\mu f])|\leq\frac1p|\d u(Z)|^p+\frac1q{\rm ext}(|\d_{\mu_w} f|^q\circ u),
\]
which is $\geq$ in \eqref{eq:dualduz}. For the opposite inequality, notice that by the very Definition \ref{def:mwug} and Lemma \ref{le:spcount} for every $\eps>0$ we can find constants $(c_n)\subset\R$, a Borel partition $(A_n)$ of $\Omega$ and 1-Lipschitz functions $(f_n)\subset\Lip_{bs}(\Y)$ such that $|\d u(Z)|(1-\eps)\leq c_n\leq {\rm ext}([u^*\d_\mu f_n])(\d u(Z))$ $\mm$-a.e.\ on $A_n$. Then for every $n\in\N$ we put $\tilde f_n:=c_n^{p-1}f_n\in \Lip_{bs}(\Y)$ and notice that
\[
\begin{split}
{\rm ext}&([u^*\d_\mu \tilde f_n])(\d u(Z))-\frac1q{\rm ext}(|\d_{\mu_w} \tilde f_n|^q\circ u)\\
&=c^{p-1}_n{\rm ext}([u^*\d_\mu f_n])(\d u(Z))-\frac{c_n^p}q{\rm ext}(|\d_{\mu_w} f_n|^q\circ u)\geq c_n^p-\frac{c_n^p}q=\frac{c_n^p}p\geq \frac1p\big(|\d u(Z)|(1-\eps)\big)^p,
\end{split}
\]
$\mm$-a.e.\ on $A_n$, so that the conclusion follows by the arbitrariness of $\eps>0$ and $n\in\N$.
\end{proof}

\subsection{Link with Sobolev maps defined via post-composition}

There is a well-established notion of metric-valued Sobolev map based on post-composition (see \cite{HKST15} and the references therein for more on the topic and detailed bibliography), in this short section we prove some natural relation between such notion and the one studied in the rest of the paper, see Propositions \ref{prop:link1}, \ref{prop:link2} for the precise statements. Notice that both for simplicity and to keep consistency with results  presented in relevant literature, we will stick to the case $\Omega=\X$ and $p=2$, but everything can be easily adapted to cover the case of arbitrary open sets and Sobolev exponents. Finally, to handle the (lack of) integrability of some functions, we will need to deal with functions in the Sobolev class $S^2(\X)$, rather than in the Sobolev space $W^{1,2}(\X)$: functions in the former class have differential in $L^2$, but no a priori condition on their integrability is imposed. We refer to \cite{Gigli14} and references therein for all the definitions and properties of $S^2(\X)$ we shall use.

\bigskip

\begin{definition}\label{def:sobmet}Let $(\X,\sfd,\mm)$ be a metric measure space, $(\Y,\sfd_\Y,\bar y)$ a pointed complete and separable space and $u\in L^{2}(\X,\Y_{\bar y})$. We say that $u\in W^{1,2}(\X,\Y_{\bar y})$ provided for any $1$-Lipschitz function $f:\Y\to\R$ with $f(\bar y)=0$ we have $f\circ u\in W^{1,2}(\X)$ and there is a function $G\in L^2(\X)$ (not depending on $f$) such that
\begin{equation}
\label{eq:dd}
|\d(f\circ u)|\leq G\qquad\mm-a.e..
\end{equation}
It is clear that for  $u\in W^{1,2}(\X,\Y_{\bar y})$  there is a minimal, in the $\mm$-a.e.\ sense, $G\geq 0$ for which \eqref{eq:dd} holds: to avoid any risk of confusion with other forms of differential used in this paper, it will be denoted $|\d'u|$.
\end{definition}
Our first result of this section is the following:
\begin{proposition}[From Sobolev to `directionally Sobolev']\label{prop:link1} Let $(\X,\sfd,\mm)$ be a $\RCD(K,\infty)$ space, $Z$ a regular vector field on it and $(\Y,\sfd_\Y,\bar y)$ a pointed complete and separable space.

Then $W^{1,2}(\X,\Y_{\bar y})\subset \KSt Z(\X,\Y_{\bar y})$ and for every  $u\in W^{1,2}(\X,\Y_{\bar y})$ we have
\begin{equation}
\label{eq:ineqdiff}
|\d u(Z)|\leq |\d' u|\,|Z|\qquad\mm-a.e..
\end{equation}
\end{proposition}
\begin{proof} Let $f:\Y\to\R$ be 1-Lipschitz with $f(\bar y)=0$. Then we know that $f\circ u\in W^{1,2}(\X)$ and thus by the very definition of Regular Lagrangian Flow we know that $\mm$-a.e.\ the inequality
\[
|f\circ u\circ \Fl^Z_s-f\circ u\circ \Fl^Z_t|\leq \int_t^s(|\d'u|\,|Z|)\circ\Fl^Z_r\,\d r
\]
holds for a.e.\ $t,s\in[0,1]$, $t<s$. It is not hard to see that the construction in Lemma \ref{le:spcount} can be modified to ensure that $f_n(\bar y)=0$ for every $n\in\N$, so that the conclusion of the lemma and what already proved give that $\mm$-a.e.\ the bound
\[
\sfd_\Y\big(u\circ \Fl^Z_s, u\circ \Fl^Z_t\big)\leq \int_t^s(|\d'u|\,|Z|)\circ\Fl^Z_r\,\d r
\]
holds for a.e.\ $t,s\in[0,1]$, $t<s$. Since both sides of this inequality are continuous in $t,s$ with values in $L^2(\X)$ (recall also Proposition \ref{le:contlpfl}), we conclude that $\mm$-a.e.\ such bound holds for every $t,s\in[0,1]$, $t<s$.  According to point $(iii)$ in Theorem \ref{thm:baseks}, this proves that $u\in \KSt Z(\X,\Y_{\bar y})$ and together with point $(iv)$ and the last part of the statement we also get  \eqref{eq:ineqdiff}.
\end{proof}
Having clarified the relation between Sobolev spaces we turn to the one between the corresponding notions of differential. Recall that in \cite{GPS18} it has been proved that to any $u\in W^{1,2}(\X,\Y_{\bar y})$ it is canonically associated a differential $\d' u:L^0(T\X)\to {\rm Ext}\big((u^*L^0_\mu(T^*\Y))^{*}\big)$ (here we used the notation $\d'u$ in place of $\d u$ to avoid risk of confusion with previously defined differentials), let us briefly review the definition of such object.  We need the following lemma:
\begin{lemma}\label{le:compo}
Let $u\in W^{1,2}(\X,\Y_{\bar y})$, put $\mu:=u_\ast(|\d'u|^2\mm)$ and let $f\in {\rm S}^2(\Y,\sfd_\Y,\mu)$. Then there is $g'\in S^2(\X)$ such that $g'=f\circ u$ $\mm$-a.e.\ on $\{|\d'u|>0\}$ and
\begin{equation}
\label{eq:chain}
|\d g'|\leq |\d_\mu f|\circ u|\d' u|\qquad\mm-a.e..
\end{equation}
More precisely, there is  $g'\in S^2(\X)$ and a sequence $(f_n)\subset \LIP_{bs}(\Y)$ such that
\begin{equation}
\label{eq:conv}
\begin{array}{rllrll}
f_n& \to\ f\qquad& \mu-a.e.&\qquad\qquad\lip_a(f_n)& \to\  |\d_\mu f|\qquad &\text{\rm in }L^2(\mu),\\
f_n\circ u& \to\ g'\qquad& \mm-a.e.&\qquad\qquad\lip_a(f_n)\circ u|\d'u|& \to\  |\d_\mu f|\circ u|\d' u|&\text{\rm in }L^2(\mm).
\end{array}
\end{equation}
\end{lemma}
Such lemma is used to define $\d'u$ as follows:
\begin{definition}\label{def:du}Let $(\X,\sfd,\mm)$ be a metric measure space, $(\Y,\sfd_\Y,\bar y)$ a pointed complete and separable space and $u\in W^{1,2}(\X,\Y_{\bar y})$. Put $\mu:=u_*(|\d'u|^2\mm)$. Then the differential $\d' u:L^0(T\X)\to {\rm Ext}\big((u^*L^0_\mu(T^*\Y))^{*}\big)$ is defined as follows. For $Z\in L^0(T\X)$ the object $\d' u(Z)\in  {\rm Ext}\big((u^*L^0_\mu(T^*\Y))^{*}\big)$ is characterized by the following property: for any $f\in S^2(\Y,\sfd_\Y,\mu)$ and $g$ as in Lemma \ref{le:compo}, it holds
\[
{\rm ext}\big([u^*\d_\mu f]\big)(\d' u(Z))=\d g(Z)
\]
\end{definition}
In \cite{GPS18} it has been proved that this definition is well posed and that the value of $|\d' u|$ is unambiguous, i.e.\ the quantity given by Definition \ref{def:sobmet} coincides with the pointwise norm of $\d' u$.

We then have the following compatibility statement:
\begin{proposition}[From differential to `directional derivative']\label{prop:link2}
With the same assumptions and notation of Proposition \ref{prop:link1} we have
\[
\d u(Z)=\d'u(Z).
\]
and in particular $|\d u(Z)|=|\d'u(Z)|$ holds $\mm$-a.e..
\end{proposition}
\begin{proof} Let $f\in L^\infty\cap W^{1,2}(\Y,\sfd_\Y,\mu)$ be arbitrary.  By Definitions \ref{def:duz} and \ref{def:du} it is sufficient to prove that any function $g'$ associated to such $f$ by Lemma \ref{le:compo} can be chosen as $g$ in Proposition \ref{prop:chain}. By \eqref{eq:ineqdiff} we know that $\{|\d u(Z)|>0\}\subset \{|\d' u|>0\}$ and thus we know that $g'$ agrees with $f\circ u$ $\mm$-a.e.\ on $\{|\d u(Z)|>0\}$. Also, the simple lemma below ensures that $t\mapsto (g\circ\Fl^Z_t-g)$ is $C^1$ with values in $L^2(\X)$, so that to conclude it is sufficient to prove the bound \eqref{eq:boundder0}. This follows by picking  $(f_n)$ as in Lemma \ref{le:compo} above and noticing that this choice is viable in the proof of Proposition \ref{prop:chain}, so that the conclusion follows by repeating verbatim the arguments already present in such proof.
\end{proof}

\begin{lemma}
Let $(\X,\sfd,\mm)$ be a $\RCD(K,\infty)$ space, $Z$ a regular vector field on it and $g\in S^2(\X)$. Then the map $[0,1]\ni t\mapsto (g\circ\Fl^Z_t-g)$ belongs to $C^1([0,1],L^2(\X))$.
\end{lemma}
\begin{proof}
It is sufficient to prove that such curve is absolutely continuous with values in $L^p(\X)$ and with derivative given by $\d g(Z)\circ\Fl^Z_t$ as then the fact that such derivative is continuous in $L^p(\X)$ (Proposition \ref{le:contlpfl}) gives the claim.

Now notice that the very definition of Regular Lagrangian Flow and Lemma \ref{le:basequ2} ensure that this is true if $g\in W^{1,2}(\X)$. For the general case put $g_n:=n\wedge g\vee(-n)$ for $n\in\N$ and $g_{n,R}:=\nchi_R g_n$, where $\nchi_R(x):=(1-\sfd(x,B_R(\bar x)))^+$ for any $R>0$, with $\bar x\in\X$ being a given, fixed point. Then conclude noticing that $\lim_n\lim_R g_{n,R}(\Fl^Z_t(x))=g(\Fl^Z_t(x))$ for any $t,x$ and that for $\mm$-a.e.\ $x\in\X$ the curves $t\mapsto \d g_{n,R}(Z)(\Fl^Z_t(x))$ converge to $t\mapsto \d g(Z)(\Fl^Z_t(x))$ in $L^2(0,1)$ when we first let $R\to\infty$ and then $n\to\infty$. This is sufficient to ensure that point $(iii)$ in Lemma \ref{le:basequ2} holds with $H_t(x):=\d g(Z)(\Fl^Z_t(x))$, which was the claim.
\end{proof}

\subsection{The case of $\Y$ universally infinitesimally Hilbertian}

We already know that $|\d u(Z)|$ satisfies a natural triangle inequality. Here we ask whether in the case $p=2$ it also holds the parallelogram identity
\begin{equation}
\label{eq:par}
|\d u(Z_1+Z_2)|^2+|\d u(Z_1-Z_2)|^2=2\big(|\d u(Z_1)|^2+|\d u(Z_2)|^2\big)
\end{equation}
in $\Omega$ provided  $u\in \KSt {Z_1}(\Omega,\Y_{\bar y})\cap \KSt {Z_2}(\Omega,\Y_{\bar y})$. Having in mind the smooth category we see that the answer must depend on $\Y$ being somehow Hilbert on small scales: if, say, $\Y$ is $\R^d$ equipped with some norm $\|\cdot\|$, $\X$ is the Euclidean space and $u,Z$ are smooth,  then $|\d u(Z)|(x)=\|\d u(Z)(x)\|$ a.e.. Hence \eqref{eq:par} holds if and only if the norm $\|\cdot\|$ comes from a scalar product.

\bigskip

For metric measure spaces $(\Y,\sfd_\Y,\mm_\Y)$ a notion of `being Hilbert on small scales' has been proposed in \cite{Gigli12}, the requirement being that $W^{1,2}(\Y,\sfd_\Y,\mm_\Y)$ is a Hilbert space (in general it is only Banach). In our setting there is no measure assigned a priori on $\Y$, but actually, as seen in the previous section, each map $u\in \KS Z(\Omega,\Y_{\bar y})$ induces its own measure on $\Y$. We are therefore lead to:
\begin{definition}[Universally infinitesimally Hilbertian]\label{def:uih} Let $(\Y,\sfd_\Y)$ be a complete and separable metric space. We say that it is universally infinitesimally Hilbertian provided for any Radon measure $\mu$ which gives finite mass to bounded sets the space $(\Y,\sfd_\Y,\mu)$ is infinitesimally Hilbertian, i.e.\ $W^{1,2}(\Y,\sfd_\Y,\mu)$ is Hilbert.
\end{definition}
It is not trivial to check that a metric space is  universally infinitesimally Hilbertian space. The first result in this direction has been obtained in \cite{GP16-2}, where it has been proved the `base case'  that $\R^d$ equipped with the Euclidean norm has such property. This result has been vastly generalized in \cite{DMGSP18} where it has been proved that spaces which are locally $\CAT(k)$, and in particular $\CAT(0)$ spaces, are universally infinitesimally Hilbertian.

\bigskip

A duality argument based on Corollary \ref{cor:dualnorm} and the linearity property \eqref{eq:lindiff} allow to get the parallelogram identity for targets which are infinitesimally Hilbertian:
\begin{theorem}\label{thm:parid}
Let $K\in\R$, $(\X,\sfd,\mm)$ be $\RCD(K,\infty)$ space, $\Omega\subset\X$ open and $Z_1,Z_2$ two regular vector fields on it. Let $(\Y,\sfd_\Y,\bar y)$ be a pointed universally infinitesimally Hilbertian space and $u\in \KSt {Z_1}(\Omega,\Y_{\bar y})\cap \KSt {Z_2}(\Omega,\Y_{\bar y})$.

Then
\[
|\d u(Z_1+Z_2)|^2+|\d u(Z_1-Z_2)|^2=2\big(|\d u(Z_1)|^2+|\d u(Z_2)|^2\big)\qquad\mm-a.e.\ on\ \Omega.
\]
\end{theorem}
\begin{proof} Recall that by Propositions \ref{prop:mult} and \ref{prop:triangle} we know that $u\in \KSt {Z_1+Z_2}(\Omega,\Y_{\bar y})\cap \KSt {Z_1-Z_2}(\Omega,\Y_{\bar y})$ so that the statement makes sense. Put for brevity $Z_+:=Z_1+Z_2$, $Z_-:=Z_1-Z_2$, define
\[
w:=\max\big\{|\d u(Z_1)|^2,|\d u(Z_2)|^2,|\d u(Z_+)|^2,|\d u(Z_-)|^2\big\}
\]
and $\mu_w:=u_*(w\mm\restr\Omega)$ and similarly $\mu_i:=u_*(|\d u(Z_i)|^2\mm)$ for $i\in\{1,2,+,-\}$. By formula \eqref{eq:dualduz} applied with $p=q=2$, $\mu:=\mu_1$ and $\mu_w:=\mu_w$ we get
\[
\frac12|\d u(Z_1)|^2=\esssup_{f \in\Lip_{bs}(\Y)}\  {\rm ext}([u^*\d_{\mu_1} f])(\d u(Z_1))-\frac12{\rm ext}\big(|\d_{\mu_w}f|^2\circ u\big ),
\]
Since an analogous formula holds for $Z_2$ we obtain
\begin{equation}
\label{eq:firststep}
\begin{split}
2|\d u(Z_1)|^2+2|\d u(Z_2)|^2=\esssup_{f,g \in\Lip_{bs}(\Y)}\  4\,{\rm ext}([u^*\d_{\mu_1} f])&(\d u(Z_1))+ 4\,{\rm ext}([u^*\d_{\mu_2} g])(\d u(Z_2))\\
&-2\,{\rm ext}\big((|\d_{\mu_w}f|^2+|\d_{\mu_w}g|^2)\circ u\big ).
\end{split}
\end{equation}
Now notice that   Proposition \ref{prop:lindif}  gives
\[
\begin{split}
4\,{\rm ext}([u^*\d_{\mu_1} f])(\d u(Z_1))+ &4\,{\rm ext}([u^*\d_{\mu_2} g])(\d u(Z_2))\\
&=2\,{\rm ext}([u^*\d_{\mu_+} (f+g)])(\d u(Z_+))+ 2\,{\rm ext}([u^*\d_{\mu_-}(f- g)])(\d u(Z_-))
\end{split}
\]
for any $f\in\Lip_{bs}(\Y)$ and that since $(\Y,\sfd_\Y,\mu_w)$ is infinitesimally Hilbertian it holds
\[
2|\d_{\mu_w} f|^2+2|\d_{\mu_w} g|^2= |\d_{\mu_w} (f+g)|^2+ |\d_{\mu_w} (f-g)|^2.
\]
Using these two identities in \eqref{eq:firststep} we deduce
\[
\begin{split}
2|\d u(Z_1)|^2+2|\d u(Z_2)|^2=&-\esssup_{f,g\in \Lip_{bs}(\Y) }{\rm ext}\big( (|\d_{\mu_w} (f+g)|^2+ |\d_{\mu_w} (f-g)|^2)\circ u\big)\\
&+\qquad 2\,{\rm ext}([u^*\d_{\mu_+} (f+g)])(\d u(Z_+))+ 2\,{\rm ext}([u^*\d_{\mu_-}(f- g)])(\d u(Z_-))
\end{split}
\]
and since as $f,g$ range over $[\Lip_{bs}(\Y)]^2$ the functions $f+g,f-g$ also range over $[\Lip_{bs}(\Y)]^2$, using again the duality formula \eqref{eq:dualduz} for $p=q=2$ and for the vectors $Z_+,Z_-$ we conclude.
\end{proof}

\def\cprime{$'$} \def\cprime{$'$}

\end{document}